\documentclass[11pt,reqno]{amsart}
\usepackage{fullpage}
\usepackage{amsmath}
\usepackage{amsfonts}
\usepackage{amssymb}
\usepackage{amssymb}
\usepackage{graphicx}
\usepackage{mathrsfs}
\usepackage{epsfig}
\usepackage{amsthm}
\usepackage{verbatim}
\usepackage{graphicx}
\usepackage{color}
\usepackage{url}

\usepackage[colorlinks,citecolor=black,linkcolor=black,
            bookmarksopen,
            bookmarksnumbered
           ]{hyperref}

\usepackage{color}

 \newcommand{\red}{\color{black}}
 \newcommand{\black}{\color{black}}

 \newcommand{\beq}{\begin{equation}}
 \newcommand{\eeq}{\end{equation}}
 \newcommand{\e}{\mathrm{e}}

  \newcommand{\Osc}{\text{Osc}}
  
    \newcommand{\oD}{\overline{\mathcal{L}}}
    
        \newcommand{\cnab}{\langle \nabla\rangle_m}

\newtheorem{theorem}{Theorem}

\newtheorem{lemma}[theorem]{Lemma}
\newtheorem{cor}[theorem]{Corollary}
\newtheorem{proposition}[theorem]{Proposition}

\theoremstyle{definition}

\newtheorem{ass}{Assumption}[section]

\newtheorem{rem}[theorem]{Remark}

\author{Fr\'ed\'eric Rousset}
\address{Laboratoire de Math\'ematiques d'Orsay (UMR 8628), Universit\'e Paris-Saclay,  91405 Orsay Cedex, France (F. Rousset)}
\email{frederic.rousset@universite-paris-saclay.fr}

\author{Katharina Schratz}
\address{LJLL (UMR 7598), Sorbonne Universit\'e, UPMC, 4 place Jussieu, 75005, Paris, France (K. Schratz)}
\email{katharina.schratz@ljll.math.upmc.fr}

\begin{document}
\begin{abstract}
We introduce a new general framework for the approximation of evolution equations at low regularity
 and  develop a new class of schemes for a  wide range of equations under  lower regularity assumptions than classical methods require.  In contrast to previous works, our new framework allows a \emph{unified} practical formulation and  the construction of the new schemes does not rely on any Fourier based expansions. This allows us  for the first time to overcome  the severe restriction to periodic boundary conditions, to embed in the same framework parabolic and dispersive equations and to handle
 nonlinearities that are not polynomial. In particular, as our new formalism does no longer require periodicity of the problem, one may couple the new  time discretisation technique not only with spectral methods, but rather with various spatial discretisations.  
  We  apply our general theory to the time discretization of various concrete PDEs, such as the nonlinear  heat equation, the  nonlinear Schr\"odinger equation, the complex Ginzburg-Landau equation, the half wave and Klein--Gordon equations,  set 
 in  $\Omega \subset \mathbb{R}^d$, $d \leq 3$  with suitable boundary
conditions.
\end{abstract}

\title[]{A general framework of low regularity integrators}

\maketitle

\section{Introduction}
We consider a general class of evolution equations under the form
\begin{equation}\label{ev}
 \partial_t u - \mathcal{L} u = f(u,\overline u) \qquad (t,x) \in \mathbb{R}\times \Omega
\end{equation}
with 
$\Omega \subset \mathbb{R}^d$. We add an initial condition
\begin{equation}
\label{init}
u_{/t=0}= u_{0}
\end{equation}
 and when $\partial \Omega \neq \emptyset$ some appropriate homogeneous  boundary conditions.
  The unknown $u$ can be complex valued or real-valued. The precise assumptions for  the general problem \eqref{ev} will  be given in Section \ref{introAss} and concrete examples that  fit the assumptions are  illustrated in Section \ref{introEx}.

In the last decades, a large variety of  discretisation techniques  was introduced  for the time resolution of evolution equations of type \eqref{ev} reaching from splitting  methods over exponential   integrators up to Runge--Kutta and Lawson   schemes \cite{Faou12,HNW93,HLW,HochOst10,HLO20,HLRS10,LR04,McLacQ02,SanBook}.  
 While  such \emph{classical} discretisation techniques  in general  provide  a good approximation to smooth solutions of partial differential equations (PDEs),   they often drastically break down at low regularity:   rough  data and high oscillations can cause  severe order reduction leading to loss of convergence and huge computational costs.

Lack of smoothness  is thereby mostly  negligible for parabolic problems:  thanks to the parabolic smoothing property the solution is regularised away from time $t =0$ such that classical schemes in general provide a good approximation  as time evolves.  This stands in stark contrast to the dispersive setting where no pointwise smoothing can be expected. Rough (or oscillatory)  data spreads in time and in space  which causes the breakdown of classical approximation techniques.   The control of nonlinear terms in dispersive PDEs at low regularity is thus an ongoing challenge in computational mathematics at large.


Recently, a new type of discretisation (so-called resonance based discretisation) was introduced  for various dispersive equations with periodic boundary conditions (that is $\Omega = \mathbb{T}^d$) when the linear part is defined by a  differential operators $\mathcal{L} =\sum_{i=1}^d \partial_{x_i}^n$ with $n \in \mathbb{N}$.  First   for  the periodic Korteweg--de Vries equation (\cite{HS16}), then for periodic Schr\"odinger equations  (\cite{OS18})   and lately for other periodic dispersive equations such as Boussinesq  (\cite{OstSu19}) and Dirac equations (\cite{SWZ20}). The main idea behind the new discretisation technique is  the following:  Instead of  discretising the evolution equation \eqref{ev}  directly with classical techniques based on a Taylor series expansion of the exact solution, one embeds the underlying structure of  nonlinear frequency interactions   into the numerical discretisation. The  latter is achieved by expanding the exact solution  into a Fourier series expansion $u(t,x) = \sum_{k \in \mathbb{Z}^d} \hat{u}_k(t) e^{i k x}$.
In the discretisation of the equation written on the  Fourier side  one can then (easily) tackle the  nonlinear interactions of the Fourier modes   and embed their dominant parts into the numerical discretisation. 
In many cases this enables us to solve the dispersive PDE under (much) lower regularity assumptions than required by  classical numerical schemes  which are based on linearised frequency approximations. For an extensive overview on the comparison of this  approach to classical  methods,  we refer to  \cite{BS20}.  

While the novel  approach allowed us to solve dispersive PDEs in a more general setting, i.e., under lower regularity assumptions, its main drawback lies in the fact that each and every equation  has to be carefully analyzed on its own since   the 
nonlinear frequency interaction  strongly \black differs from PDE to PDE.
 With the aid of decorated tree series we could recently show that  a general high-order framework of resonance based 
discretisations does exist \cite{BS20}.  However, in order to obtain a practical implementation  each scheme has to be derived separately through  involved Butcher-Connes-Kreimer Hopf algebras and their corresponding co-products.


Consequently, while  resonance based approaches allow us to solve dispersive PDEs under (much) lower regularity assumptions than classical schemes require,  they   face {three main obstacles}: 
\begin{itemize}
\item[(A)]   \vskip0.05cm They are restricted to  {periodic boundary conditions} due to their strong dependence on the Fourier series expansion of the exact solution.
\item[(B)]   They are limited to  {classical differential operators} of type $\mathcal{L}= \sum_{i=1}^d\partial_{x_i}^n $ with $n \in \mathbb{N}$, e.g., 
$$\mathcal{L} = - \Delta, \partial_x, \partial_x^3, \ldots, $$
in order to involve only  polynomials in the frequencies in the interaction phase and to allow  the dominant parts to be straightforwardly  extracted. The latter, however, excludes important classes of equations, such as  Klein--Gordon, or wave  type systems  which does not  involve polynomials, but rather interactions of square roots of  frequencies (for example for quadratic nonlinearities $\sqrt{(j+k)^2+m^2} - \sqrt{j^2+m^2}- \sqrt{k^2+m^2}$, $\vert j+k \vert - \vert j \vert  -\vert k \vert$, etc).
\item[(C)]  The  {lack of a unified practical formulation} of the schemes which is the main advantage of Runge--Kutta methods, splitting methods, exponential integrators and other classical discretisation techniques.
\end{itemize} \vskip0.1cm

\subsection{Aim of the paper}
 In this work we present a novel framework  of low regularity integrators \black which allows us to overcome the above obstacles. \red
 The central idea lies in embedding the underlying oscillatory structure of  \eqref{ev} into our numerical discretisation (without employing any Fourier based techniques). This is achieved by introducing  {filtered oscillations} (cf. \eqref{filtG}) which allow us to treat the dominant oscillations, triggered by the operator $-\mathcal{L} + \overline {\mathcal{L}}$ exactly, while only approximating the lower order parts by a {stabilised} Taylor series expansion. \black This framework yields a new class of schemes that can be used   on a wide range of equations under  lower regularity assumptions than classical methods require.  \black The new framework \black allows for a unified practical formulation (C) while still embedding the central oscillations in the discretisation. The construction of the new schemes does not rely on any Fourier based expansions. This will allow us to solve a large class of equations at low regularity without requiring periodicity of the problem (A). In particular,  as the spatial domain is no longer restricted to periodicity, one may couple the new  time discretisation technique not only with spectral methods, but rather with various spatial discretisations (e.g., finite differences, finite element methods). \red The practical   implementation on general domains will require, as for classical splitting and exponential integrators methods, suitable Krylov space methods in order to approximate the exponential $e^{t\mathcal{L}}$, respectively, the action of the $\varphi-$ functions.  \black  A fully discrete analysis with various spatial discretisation methods is plan of future research. In addition of overcoming periodicity, the new general framework  also allows us for the first time  to   deal in the same formulation  with parabolic and   dispersive or  hyperbolic  problems (B). The latter will open up this new low regularity time discretisation  framework to a larger class of equations   including for instance  heat equations, reaction-diffusion problems and wave-type systems. Furthermore, in contrast to previous works, our framework is no longer restricted to strict polynomial nonlinearities and we can treat for instance the Sine--Gordon equation with our new framework.  We   carry out   an overarching \emph{abstract} convergence analysis  for our new class of schemes in the general setting~\eqref{ev}. The  abstract error bounds we establish apply to a lot of examples.  We illustrate them for  nonlinear heat,  Schr\"odinger and Ginzburg-Landau type equations as well as  the half wave,  Klein--Gordon and wave equations.\\

The new first- and second-order schemes together with their main convergence result  are presented in Section \ref{sec:duhi}. In Section \ref{sec:1} and \ref{sec:2} we carry out an expansion of the underlying oscillations up to first- and second-order, respectively. These expansions together with their improved local error structure motivate  the new numerical schemes. We focus on first- and second-order methods. However, our framework can   be extended to higher order.   The   error analysis at first- and second-order is presented in Section \ref{locErr1} and \ref{locErr2}, respectively. Finally,  in Section \ref{sectionexamples} we illustrate our convergence analysis on various examples.  

 \subsection{Assumptions}\label{introAss}
  The linear  operator $ \mathcal{L}$ is defined on a Hilbert space
   $X$  of complex valued  functions $u \in \mathbb{C}$ with norm denoted by $\| \cdot \|$ and domain $\mathcal{D}(\mathcal{L})$.
    We assume that the complex conjugation $u \mapsto \overline{u}$ is an isometry on $X$.
     When $\partial \Omega \neq 0$, the boundary conditions
   will be  encoded in the choice of the domain of the operator $\mathcal{L}$.  We shall also use  the operator $ \overline{\mathcal{L}}$
    defined by $  \overline{\mathcal{L}} u = \overline{ \mathcal{L} \overline{u}}.$
    In the following we assume that the  nonlinearity  $f$ is tensorized  under the form
\begin{equation}
\label{nonlin}
f(v, w) = \mathcal{B}\left(F(v) \cdot G(w)\right), \quad  F, \, G:\, \mathbb{C} \rightarrow \mathbb{C}^J, 
\end{equation}
where we use  the  notation 
 $ X\cdot Y = \sum_{k}X_{k}Y_{k}, \quad X, \, Y \in \mathbb{C}^J$, and    $\mathcal{B}$ is a linear operator.

\begin{ass}\label{ass1}
Our main assumptions on $\mathcal{L}$ are the following 
\begin{enumerate}
\item[i)] $ \mathcal{L}$ generates a strongly continuous semigroup $\{e^{t \mathcal{L}}\}_{t\geq 0}$ of contractions  on $X$;
\item[ii)] $\mathcal{A}= -\mathcal{L}+\overline{\mathcal{L}}$ generates a group $\{e^{t \mathcal{A}}\}_{t\in \mathbb{R}}$  of unitary operators on X;
\item[iii)] $\mathcal{L}$ and $\overline{\mathcal{L}}$ commute: $[\mathcal{L}, \overline{\mathcal{L}} ] = 0 $.
\end{enumerate}
\end{ass}
The above assumptions allow us to deal in an unified framework  with parabolic,  dispersive as well as mixed equations, we shall give examples in Section \ref{introEx} below.
Note that the second assumption is automatically verified  if $\mathcal{A}= 0$ so in particular when $\mathcal{L}$ is real
 for real functions.
 
 The consequences of the above assumptions that we shall use are the following. Since $\{e^{t \mathcal{L}}\}_{t\geq 0}$
  is  a strongly continuous semigroup of contractions \red it holds that\black
   $$\red \| e^{t \mathcal{L}}  \|_{\mathcal{L}(X)} \leq  1.$$
    We can  define powers of $\mathcal{L}$, $\mathcal{L}^\alpha$ for $\alpha \in \mathbb{R}$ and we shall denote 
     the domain  $\mathcal{D}(\mathcal{L}^\alpha)$ by $X^\alpha$ (with the convention that $X^0= X$), the graph norm will be denoted 
      for $\alpha \geq 0$ by 
     $$ \|u\|_{\alpha}= \|u\| + \|\mathcal{ L}^{\alpha} u \|.$$
  We could  also deal  with general strongly continuous semigroups which then satisfies
  $$ \| e^{t \mathcal{L}} u \| \leq  M e^{ \omega t}$$
  for some $\omega \geq 0$ but this is not needed for the examples that we shall study.

\begin{ass}\label{assnon}
 To handle the ``nonlinearity" $f$, we assume that 
 there exists $\, a_{1}>a_{0} \geq 0$ such that   for every $\alpha \in [a_{0}, a_{1}],$
  \begin{equation}
  \label{picard}  \|f(u, \overline{u}) \|_{\alpha} \leq C_{\alpha}( \|u\|_{a_{0}}) \|u \|_{\alpha} , \quad 
         \| f(u,\overline{u})-  f(v,\overline{v}) \|_{a_{0}} \leq  C_{a_{0}}( \|u\|_{a_{0}}, \|v\|_{a_{0}})  \|u-v \|_{a_{0}}
         \end{equation}
         where $C_{\alpha}( \cdot)$ denotes a  continuous non-decreasing functions of its arguments. 
         
 \red{ In the above assumption, $a_{0}$ is the minimal regularity needed to have a standard well-posedness result, while
  $a_{1}$ stands for the maximal regularity that we can propagate without imposing further compatibility conditions. 
  The threshold $a_{1}$ will be meaningful only when we consider domains with boundaries, on the whole space  or the torus
   $a_{1}$ can be arbitrarily large.}

\end{ass}

%


\subsection{Examples}\label{introEx}

The main examples which will be covered by the above framework are 
\begin{itemize}
\item Nonlinear heat equations
$$ \partial_t u  -  \Delta u =f(u,\overline u),\quad \text{i.e., }\quad \mathcal{L}=\Delta,  \, \mathcal{A}= 0,$$
where $f$ is  sufficiently smooth\black
\item Nonlinear Schr\"odinger equations
$$
 i \partial_t u +  \Delta u = \pm \vert u\vert^{2m} u , \,m \in \mathbb{N}\black\quad \text{i.e., }\quad  \mathcal{L}= i \Delta, \,  \mathcal{A}= -  2 i \Delta , \,     \quad f(u,\overline u) = \pm  i u^{m+1 } \overline u^{m}.$$
 \item  Complex Ginzburg Landau equations
\begin{align*} & \partial_t u  -  \alpha  \Delta u = \gamma  u (1 - |u|^2), \,  \alpha, \, \gamma  \in \mathbb{C}, \mbox{Re }\alpha \geq 0, \\  &\text{i.e., }\quad \mathcal{L}=\alpha\Delta, \,  \, \mathcal{A}=  - 2 i \mbox{Im } \alpha \Delta, \, f(u, \overline{u})= \gamma  u (1 - u \overline{u})
\end{align*}
\item Half wave equation
$$ i \partial_t u + \vert \nabla \vert u =  \pm |u|^2 u,\quad \text{i.e., }\quad \mathcal{L}= i \vert \nabla \vert, \,   \quad f(u,\overline u) = \pm u^2 \overline{u}$$
\item  Klein--Gordon and wave-type equations reformulated as first-order systems
$$  \partial_{tt} u -  \Delta u  +  m^2 u =  f(u) $$
with $f$ smooth enough, for example, $f(u)= u^p$ or   Sine-Gordon equations  $f(u) = \mathrm{sin}(u)$.
\end{itemize}
We shall apply our general theory to the time discretization of the above partial differential equations set 
 in  $\Omega \subset \mathbb{R}^d$, $d \leq 3$  with suitable boundary
conditions. This will be the aim of the last part of the paper, see Section \ref{sectionexamples}.

\section{Main results}\label{sec:duhi}
We shall deal with mild solutions of \eqref{ev}, \eqref{init} which satisfy the Duhamel formula
\begin{align}\label{duh}
u(t) = e^{ t \mathcal{L}} u_0 +   
\int_0^t e^{ (t-  \xi) \mathcal{L}}  f \left(u(\xi), \overline u(\xi) \right) d\xi.
\end{align}
By using  Assumptions \ref{ass1}, \ref{assnon}, we can easily get from a Banach fixed point argument the following local existence result:
\begin{theorem}
\label{localwp}
Under the Assumptions \ref{ass1}, \ref{assnon}, for every $u_{0} \in X^{a_{0}}$, there exists $T>0$ and a unique solution
 $u \in \mathcal{C}([0, T], X^{a_{0}})$ solution of \eqref{duh}. Moreover if $u_{0} \in X^{\alpha}$, $\alpha \in (a_{0}, a_{1}]$, then
we also have   $u \in \mathcal{C}([0, T], X^{\alpha}).$
\end{theorem}

The aim of this paper is to introduce low-regularity  integrators for the approximation of  Duhamel's formula~\eqref{duh}. We shall call these schemes \emph{Duhamel integrators}.

\subsection{First-order Duhamel integrator}
 At first-order our new Duhamel integrator   takes the form for $l \geq 0$
\begin{equation}\label{num1Intro}
u^{\ell+1}=  \Phi_{\text{num},1}^\tau({u^\ell}) = e^{\tau \mathcal{L}} \Big( u^\ell +  \tau \mathcal{B} \left( F(u^\ell) \cdot \varphi_1\big(\tau  \mathcal{A}\big) G(\overline {u^\ell})\right)\Big) \quad\text{with} \quad \varphi_1(z) = \frac{e^{z}-1}{z}
\end{equation}
and we set $u^0= u_{0}$.
Let us recall that we have set  $\mathcal{A}= - \mathcal{L}+ \oD$. Thanks to Assumption \eqref{ass1}, Stone's Theorem implies 
 that  $\mathcal{A}$ is under the form $i \mathcal{M}$ with $\mathcal{M}$ self-adjoint. We can therefore define  $\varphi_1\big(\tau  \mathcal{A}\big)$
  by using the functional calculus for self-adjoint operator, this yields a bounded operator since $\varphi_{1}$ is bounded on 
   $i \mathbb{R}$.
  
  It will be convenient to set
  \begin{equation}
  \label{Psinum1}
  \Phi_{\text{num},1}^\tau({u^\ell}) = e^{\tau \mathcal{L}} \Big( u^\ell +  \tau \mathcal{B} \left( F(u^\ell) \cdot \varphi_1\big(\tau  \mathcal{A}\big)G(\overline {u^\ell})\right)\Big)= e^{\tau \mathcal{L}} \left( u^\ell + \tau  \Psi^\tau_{\text{num}, 1}(u^l)\right).
  \end{equation}
Let us define for a function  $H(v_{1}, \cdots v_{n})$, $n \geq 1$ and a linear operator $L$
the commutator type term 
 $ \mathcal{C}[H, L]$    by
\begin{align}\label{DlowIntro}
\mathcal{C}[H, L](v_{1}, \cdots, v_{n})
= 
- L ( H \left( v_{1}, \cdots    v_{n}\right))
+  \sum_{i=1}^n   \,  D_{i} H( v_{1}, \cdots,  v_{n}) \cdot  L v_{i}
\end{align}
where $D_{i} H$ stands for the partial differential of $H$ with respect to the variable $v_{i}.$
The estimate of this type of terms will be crucial in order to estimate the local error of our schemes.
From Assumption \ref{assnon}, if $a_{1}$ can be taken bigger than $1$,  we can get that 
\begin{equation}
\label{trivial}  \|
\mathcal{C}[f, \mathcal{L}] (v, w) \| \leq C( \|v\|_{a_{0}},   \|w\|_{a_{0}}) (\|v\|_{1}+  \|w\|_1).
\end{equation}
Nervertheless, this estimate is very crude since each term in the expression of $\mathcal{C}[f, \mathcal{L}]$
actually satisfies the above estimate. By using this estimate in the analysis of our first-order scheme
we would get first-order convergence in $X$  for data in $X^1$ without any improvement compared to  classical schemes.
However, if  $\mathcal{L}$ is a differential operator of order $m$, we can  get a  better estimate 
 by using the Leibnitz formula which implies that $ \mathcal{C}[f, \mathcal{L}] $ actually involves at most
  only $m-1$ derivatives of $u$ so that a better estimate can be expected.
  If $\mathcal{L}$ is not a differential operator but satisfies a generalized Leibnitz rule, such an improvement
  can also be obtained as we will see on the examples.
  
 We shall now state an  abstract assumption about the estimate of these commutator terms which will fix the needed regularity for the convergence of the new scheme \eqref{num1Intro}. In order to get a result as general as possible, we take some  $\alpha_{0} \in [0, a_{1}]$ and we will measure errors in $X^{\alpha_{0}}$.  

\begin{ass}\label{asscom}
 There exists $\alpha_{1} \in (a_{0}, a_{1}] $ such that for every $v$, $w \in X^{\alpha_{1}}$, 
\begin{align}
\label{C1} &  \|\mathcal{C}[f, \mathcal{L}] (v, w) \|_{\alpha_{0}} \leq C_{\alpha_{0}}( \|v\|_{\alpha_{1}},   \|w\|_{\alpha_{1}}), \\
\label{C2} &  \|\mathcal{B}( F(v) \cdot e^{s \mathcal{A}}\mathcal{C}[G, \mathcal{A}] (w) )\|_{\alpha_{0}} \leq C_{\alpha_{0}}( \|v\|_{\alpha_{1}},   \|w\|_{\alpha_{1}}), 
 \quad \forall s \in \mathbb{R}.
\end{align}
The second estimate is uniform with respect to $s$.
\end{ass}

\red In the above assumption $a_{1}$ stands for the minimal regularity needed to estimate the commutator terms (and thus the local
error as we shall see) in the space that we have chosen to measure the error which is  $X^{\alpha_{0}}$.
\black

Note that if we choose to measure errors in $X$ so that $\alpha_{0}= 0$, then if we can check on a concrete example  that $\alpha_{1}<1 $ we obtain indeed
 an improvement on the trivial estimate \eqref{trivial}. 
Due to the favorable local error structure of our scheme which involves these commutator type terms, this
will allow us to get first-order convergence in $X$  for data in $X^{\alpha_{1}}$ instead of $X^1$.

\red To make our point more concrete let us consider the simple case $\mathcal{L}=  \Delta$ on the torus $\mathbb{T}^d$, 
and $f(u)= u^2$. Note that $X^\alpha = H^{2 \alpha}$ in this case. We  have the explicit formula
$$   \mathcal{C}[f, \mathcal{L}] (v, w)= -2 \sum_{k=1}^d \partial_{k}v \partial_{k} w.$$
  If we choose to measure the error in $L^2$, so that $\alpha_{0}=0$,  we need $ \partial_{k}v \partial_{k} w \in L^2$. If $d \leq 4$, 
   we can take 
   $v,\, w \in W^{1, 4}$ and thus, since  $H^{d\over 4} \subset L^4$, this yields $v\, w \in H^{ 1 + {d \over 4}}$. This means
    that we can take $a_{1}= {1 \over 2} + {d \over 8}$ which is strictly smaller than $1$ for $d \leq 3$.
    \black

Before stating our main result for first-order convergence, our last technical Assumption is the following
\begin{ass}\label{asstech}
Let $\alpha_{0}$, $\alpha_{1}$ be  given as above, we assume that
\begin{equation}
\label{tech0}
  \| f(u,\overline{u})-  f(v,\overline{v}) \|_{\alpha_{0}} \leq  C_{\alpha_{0}}( \|u\|_{a_{0}}, \|v\|_{a_{0}}) \|u-v\|_{\alpha_{0}}, \quad
 \| f(u,\overline{u}) \|_{\alpha_{0}}  \leq C_{\alpha_{0}}(  \|u\|_{\alpha_{1}}),
 \end{equation}
 and that the numerical flux is such that
 \begin{align}
&  \label{tech0bis}
  \| \Psi^\tau_{\text{num,1}} (u) \|_{\alpha_{1}}
   \leq  C_{\alpha_{1}}(  \|u\|_{\alpha_{1}}), \\
\label{tech1}  &   \| \Psi_{\text{num}, 1}^\tau(u) -  \Psi_{\text{num}, 1}^\tau({\red v}) \|_{\alpha_{0}}  
 \leq  C_{\alpha_{0}}( \|u\|_{a_{0}}, \|v\|_{a_{0}} ) \|u- v \|_{\alpha_{0}} \\ 
\label{tech2}   
&  \| \Psi_{\text{num}, 1}^\tau(u) -  \Psi_{\text{num}, 1}^\tau({\red v}) \|_{a_{0}} 
\leq  C_{a_{0}}( \|u\|_{a_{0}}, \|v\|_{a_{0}} ) \|u- v \|_{a_{0}}.
 \end{align}
\end{ass}
Note that we actually need both  \eqref{tech1}, \eqref{tech2} only when $\alpha_{0}< a_{0}$. 
{\red This assumption
 will be used  to prove boundedness for the numerical solution in $X^{a_{0}}$ and convergence in $X^{\alpha_{0}}$.}

The global error estimate for the first-order Duhamel integrator \eqref{num1Intro} then takes the form.
\begin{theorem} 
\label{theo1}
\red Let   Assumptions \ref{ass1},  \ref{assnon} hold, let us choose some  $\alpha_{0}\in [0, a_{1}]$ 
such that Assumptions  \ref{asscom}, \ref{asstech} for some $\alpha_{1} \in (a_{0}, a_{1}]$
 hold. \black Then, 
 for every  $u_{0} \in X^{\alpha_{1}}$, $\alpha_{1}$ given by Assumption \ref{asscom}, let $T>0$ and $u \in \mathcal{C}([0, T], X^{\alpha_{1}})$
  the unique solution of \eqref{duh} given by Theorem \ref{localwp}, let  $u^n$ denote the numerical solution given by \eqref{num1Intro}.
   Then there exists $C_{T}>0$ such that
 $$ \|u(n \tau) - u^n \|_{\alpha_{0}} \leq C_{T} \tau, \quad 0 \leq n \tau \leq T.$$
\end{theorem}

This abstract result will be applied to the concrete examples presented in Section \ref{sectionexamples}.
We will mainly  have to check that Assumption \ref{asscom} indeed holds true for some  $ \alpha_{1}< \alpha_{0} + 1 $ to get a concrete
first-order convergence result that holds true for rougher data than classical schemes.
One of the interest of this general theory is that different types of PDEs (parabolic, hyperbolic, dispersive) can be covered
in the same unified framework. Nevertheless, the abstract proof will not use any fine structure of the PDE (for example
 smoothing effect for parabolic equations, Strichartz estimates in exterior domains  for dispersive equations, etc). 
  Using these specific properties  when it is possible  on a  concrete example would allow to reduce again
   the smoothness of the initial data, see, e.g. \cite{HochOst10,ORS19}.

In certain cases we can relate our scheme \eqref{num1Intro} to more classical schemes.

\noindent{\bf Parabolic setting.}
 In the parabolic setting  $\mathcal{A}= 0$ (or in case of  nonlinearities depending only on $u$, i.e., $f(u,\overline u) = f(u)$),  the scheme \eqref{num1Intro} collapses to the classical exponential Euler method
\begin{equation}\label{num1Par}
u^{\ell+1}=  e^{\tau \mathcal{L}} \Big( u^\ell +  \tau f(u^\ell, \overline u^\ell) \Big) .
\end{equation}
{\bf Filtered Lie splitting.}
For a  nonlinearity \eqref{nonlin} with $\mathcal{B} = 1$ and polynomial $F(v) = v^p$ 
 the low regularity scheme \eqref{num1Intro} can also be seen as a filtered Lie splitting scheme with filter function
\begin{align}\label{filterIntro}
\Psi(\tau) = \varphi_1\big(\tau  \mathcal{A}\big) .
\end{align} Indeed, the approximation 
\begin{equation}\label{numFLieIntroe}
u^{\ell+1}=  e^{\tau \mathcal{L}} \e^{  \tau  \frac{1}{p}F'(u^\ell) \cdot  \left(\Psi(\tau) G(\overline{u^\ell})\right) } u^\ell  
\end{equation}
with filter function \eqref{filterIntro} introduces a similar   error structure as the scheme \eqref{num1Intro}.  In the parabolic setting $\mathcal{A} =0 $  the filter function \eqref{filterIntro} thereby naturally reduces to $\Psi(\tau) \equiv 1$  and the filtered Lie splitting \eqref{numFLieIntroe} collapses to a classical Lie splitting. In the dispersive setting $\mathcal{L}\neq \oD$, on the other hand, the filtered splitting \eqref{numFLieIntroe} allows for an improved   error structure similar to the convergence result given in Theorem \ref{theo1}, see also Remark \ref{rem:splitNLS} for the example of the cubic Schr\"odinger equation. For details on filter functions   we refer to  \cite{HLW} and the references therein. 

\subsection{Second-order Duhamel integrator}
At second-order our new Duhamel integrator   takes the form
\begin{equation}\label{num2Intro}
\begin{aligned}
u^{\ell+1}& = e^{ \tau \mathcal{L}} u^\ell +  \tau e^{ \tau  \mathcal{L}} \mathcal{B}\left(F (u^\ell) \cdot \varphi_1\left(  \tau  \mathcal{A}\right) G(\overline u^\ell) \right)  +  
  \tau  e^{ \tau \mathcal{L}}  \mathcal{B}\left(F (u^\ell) \cdot \varphi_2\left(  \tau  \mathcal{A}\right) \delta_\tau \left( e^{- \tau \mathcal{A}}G \left(  e^{\tau \mathcal{A} }  \overline u^\ell\right) \right)    \right)\\
&
+  \tau \delta_{\xi} \left( 
e^{(\tau -\xi )\mathcal{L}}\mathcal{B} \left( F \left(e^{\xi\mathcal{L}} u^\ell\right) \cdot \varphi_2\left(\tau\mathcal{A}\right)G\left(e^{\xi \mathcal{L}} \overline u^\ell\right)\right) \right)_{/\xi=\tau} 
  + \frac{\tau^2}{2}e^{ \tau \mathcal{L}} \left(   D_1 f^\ell  \cdot  f^\ell  +  D_2 f^\ell \cdot  \overline{f^\ell}\right) \\&
  = \Phi^\tau_{\text{num}, 2}(u^\ell) = e^{\tau \mathcal{L}} \left(u^\ell + \tau  \Psi^\tau_{\text{num}, 2}(u^\ell)\right)
  \end{aligned}
\end{equation}
where $ f^\ell =  f(u^\ell,\overline u^\ell)$ and we use the notation of the standard shift operator  $$\delta_\tau g(\tau) = g(\tau) - g(0).$$
Note that in general $D_1 f \left(  u^\ell,   \overline  u^\ell\right)  $  and $ D_2 f\left(   u^\ell,   \overline  u^\ell\right)$ can be calculated analytically. Nevertheless on concrete examples,  they could also be approximated by standard finite differences. Similarly to the first-order scheme \eqref{num1Intro}, its second-order counterpart \eqref{num2Intro} introduces an improved commutator-type   error structure.

In order to analyze our second-order scheme, we need to introduce second-order commutators.
For   $H(v_{1}, \cdots v_{n})$ and a linear operator $L$, we define the iterated ``commutator"
\begin{equation}\label{CC2}
\mathcal{C}^2[ H, L ] (v_{1}, \cdots, v_{n})
 = \mathcal{C}[ \mathcal{C}[H, L], L] (v_{1}, \cdots,  v_{n}).
 \end{equation}
 
 We shall again measure the error in $X^{\alpha_{0}}$ for some given $\alpha_{0}$.
 The counterpart of Assumption \ref{asscom} will be the following  
 \begin{ass}\label{asscom2}
 There exists $\alpha_{2} \in (a_{0}, a_{1}] $ such that for every $v$, $w \in X^{\alpha_{2}}$, 
\begin{align}
\label{C12} &  \|\mathcal{C}^2[f, \mathcal{L}] (v, w) \|_{\alpha_{0}} \leq C_{\alpha_{0}}( \|v\|_{\alpha_{2}},   \|w\|_{\alpha_{2}}), \\
\label{C22} &  \|\mathcal{B}\left( F(v) \cdot e^{s \mathcal{A}}\mathcal{C}^2[G, \mathcal{A}] (w) \right)\|_{\alpha_{0}} \leq C_{\alpha_{0}}(
 \|v\|_{\alpha_{2}},   \|w\|_{\alpha_{2}}), 
 \quad \forall s \in \mathbb{R}, \\
 \label{C32} &  \|\mathcal{C}[\Psi_{s}, \mathcal{L}] (v, w) \|_{\alpha_{0}} \leq C_{\alpha_{0}}(
 \|v\|_{\alpha_{2}},   \|w\|_{\alpha_{2}}), \quad \forall s \in \mathbb{R},
\end{align}
where $\Psi_{s}(v,w)= \mathcal{B}\left( F(v)\cdot e^{s \mathcal{A}} \mathcal{C}[G, \mathcal{A}](w)\right).$
The second and third estimates are uniform with respect to $s$.
\end{ass}
Again in the case $\alpha_{0}= 0$, a rough estimate that does not use the commutator structure would allow to estimate these commutators
 for $v$, $w$ in $X^2$ which would produce second-order convergence  of the scheme in $X$ for data in $X^2$
  like other classical schemes. If we can check on a concrete example by using the commutator structure that
   $\alpha_{2}<2$, then we can get  improved second-order convergence  of the scheme in $X$ for data only  in $X^{\alpha_{2}}$.
    
 The counterpart of Assumption \ref{asstech} will be
 
 \begin{ass}\label{asstech2}
Let $\alpha_{0}$, $\alpha_{2}$ be  given as above, we assume that
\begin{align}
\label{tech02}
&  \| f(u,\overline{u})-  f(v,\overline{v}) \|_{\alpha_{0} }  \leq  C_{\alpha_{0}}( \|u\|_{a_{0}}, \|v\|_{a_{0}}) \|u-v\|_{\alpha_{0} }, \\
\label{tech0bisbis2}&  \| L^k (f(u,\overline{u})) \|_{\alpha_{0}}  + \| D_{i} f(u, \overline{u}) \cdot Lu \|_{\alpha_{0}}
 +   \|D_{i} f(u, \overline{u}) \cdot  g(v, \overline{v}) \|_{\alpha_{0}}
 \leq C_{\alpha_{0}}(  \|u\|_{\alpha_{2}}, \|v\|_{\alpha_{2}}), \\
&   \nonumber \mbox{for }
 \, k=0, \, 1, \, i=1, \, 2,  \, w\in \{ \mathcal{L}u, \, \overline{\mathcal{L}u}, \, f(v,\overline{v}), \, \overline{f}(v, \overline{v})\},  \\
 & \label{tech0bisbis3}
 \| L \left( D_{i} f(u, \overline{u}) \cdot g(v, \overline{v}) \right) \|_{\alpha_{0}}
  +  \| D_{ij}^2 f(u, \overline{u}) \cdot (w_{1}, w_{2}) \|_{\alpha_{0}} \leq C_{\alpha_{0}} (  \|u\|_{\alpha_{2}}, \|v\|_{\alpha_{2}}, 
  \| w \|_{\alpha_{2}}), \quad \\
&   \nonumber  \mbox{for } i, \, j \in \{1, 2\}, \,  w_{i} \in \{ \mathcal{L}w, \overline{\mathcal{L}} \overline{w}, f(v,v), \overline{f}(v, \overline v) \}, \,
  \end{align} 
  and for the numerical flux that 
  \begin{align}
&  \label{tech0bis2}
  \| \Psi^\tau_{\text{num,2}} (u) \|_{\alpha_{2}}
   \leq  C_{\alpha_{2}}(  \|u\|_{\alpha_{2}}),\\
&  \label{tech12}   \| \Psi_{\text{num}, 2}^\tau(u) -  \Psi_{\text{num}, 2}^\tau({\red v}) \|_{\alpha_{0}}  
  \leq  C_{a_{0}}( \|u\|_{a_{0}}, \|v\|_{a_{0}} ) \|u- v \|_{\alpha_{0}} ,  \\
&  \label{tech22}   
 \| \Psi_{\text{num}, 2}^\tau(u) -  \Psi_{\text{num}, 2}^\tau({\red v}) \|_{a_{0}} 
\leq  C_{\alpha_{0}}( \|u\|_{a_{0}}, \|v\|_{a_{0}} ) \|u- v \|_{a_{0}}. 
   \end{align} 
\end{ass}
To check the above assumptions on concrete examples,  we will be sometimes obliged to take 
 $a_{0}$ bigger than for first-order convergence, nevertheless we did not change the notation.
 Our second-order convergence result then reads:
 \begin{theorem} 
\label{theo2}
\red Let   Assumptions \ref{ass1},  \ref{assnon} hold, let us choose some  $\alpha_{0}\in [0, a_{1}]$ 
such that Assumptions  \ref{asscom2}, \ref{asstech2} for some $\alpha_{2} \in (a_{0}, a_{1}]$
 hold. \black Then, 
 for every  $u_{0} \in X^{\alpha_{2}}$, $\alpha_{2}$ given by Assumption \ref{asscom2}, let $T>0$ and $u \in \mathcal{C}([0, T], X^{\alpha_{2}})$
  the unique solution of \eqref{duh} given by Theorem \ref{localwp}, let  $u^n$ denote the numerical solution given by \eqref{num2Intro}.
   Then there exists $C_{T}>0$ such that
 $$ \|u(n \tau) - u^n \|_{\alpha_{0}} \leq C_{T} \tau^2, \quad 0 \leq n \tau \leq T.$$

\end{theorem}

In the parabolic case we can again relate our new second-order scheme to classical methods.
\begin{rem}[Parabolic case]
Note that in the parabolic case $(\mathcal{A} = 0)$ the second-order scheme~\eqref{num2Intro}  simplifies to an exponential Runge--Kutta method
\begin{equation}
\begin{aligned}
u^{\ell+1} & = e^{ \tau \mathcal{L}} u^\ell +  \tau e^{ \tau \mathcal{L}} 
f(u^\ell,\overline u^\ell) 
+  \frac{\tau}{2}e^{ \tau \mathcal{L}}  \delta_\tau e^{ -\tau \mathcal{L}}  f \left(e^{\tau\mathcal{L}} u^\ell, e^{\tau\mathcal{L}} \overline u^\ell\right)  \\& + \frac{\tau^2}{2}e^{ \tau \mathcal{L}} \left(  D_1 f \left(  u^\ell,   \overline  u^\ell\right)\cdot f(u^\ell,\overline u^\ell)   +   D_2 f\left(   u^\ell,   \overline  u^\ell\right)\cdot  \overline{f(u^\ell,\overline u^\ell)}  \right).
\end{aligned}
\end{equation}
\end{rem}
\begin{rem}
 \red The presented idea can  be extended to higher order by carrying out higher order stabilised Taylor series expansions of the filtered oscillations. For each additional order the order of the iterated commutator (cf. \eqref{CC2}) will thereby increase in the local approximation error.
\end{rem}

\section{first-order scheme}\label{sec:osc}\label{sec:1}
We will build our numerical schemes on iterations of  \eqref{duh}. In each iteration we  embed the dominant  oscillatory terms - triggered by the operator $\mathcal{L}$  - into our discretisation. 

In this Section we  give the main idea behind the expression of the  first-order Duhamel integrator \eqref{num1Intro} presented in Section \ref{sec:duhi} and estimate its local error. We start with a trivial but important  lemma on the first iteration.


\begin{lemma}[First-order iteration]\label{lem1}
We have
\begin{align}\label{u1pu}
u(t) = u_1(t) +  \mathcal{R}_{1,0}(t,u)
\end{align}
with the first-order iteration of Duhamel's formula $u_1(t)$ given by
\begin{align}\label{u1p}
u_1(t) = e^{ t \mathcal{L}} u_0 +
\int_0^t e^{ (t-  \xi) \mathcal{L}}  f \left(e^{ \xi\mathcal{L}} u_0,e^{ \xi  \overline{\mathcal{L}} }\overline u_0\right) d\xi 
\end{align}
and the remainder
\begin{align}\label{rem10}
 \mathcal{R}_{1,0}(t,u)
= 
\int_0^t e^{( t - \xi) \mathcal{L}}  \Big( f \left(u(\xi), \overline u(\xi) \right) 
- f \left(e^{ \xi\mathcal{L}} u_0, e^{ \xi  \overline{\mathcal{L}} }\overline u_0\right)
\Big)
d\xi.
\end{align}
\end{lemma}
\begin{proof}
The assertion directly follows from \eqref{duh}.
\end{proof}
The expansion \eqref{u1p} motivates the following definition, 
\begin{align}
\label{osc}\Osc(t,\mathcal{L}, v,\overline v)
=   \int_0^t e^{ (t-  \xi) \mathcal{L}}  f \left(e^{ \xi\mathcal{L}} v, e^{ \xi  \overline{\mathcal{L}} }\overline v\right) d\xi 
\end{align}
so that  $u_1(t)$ defined  in \eqref{u1p} can be expressed as
\begin{align}\label{u1}
u_1(t) = e^{\red t \mathcal{L}} u_0 +\Osc(t,\mathcal{L}, u_0 ,\overline u_0).
\end{align}
In order to allow for a low regularity approximation to $u_1(t)$  it is thus central to find a suitable discretisation of the  integral \eqref{osc}. 
For this purpose we set
\[
\mathcal{F}(t,\xi,v,\overline v) = e^{ (t-  \xi) \mathcal{L}} f \left(e^{ \xi\mathcal{L}} v, e^{ \xi \overline{\mathcal{L}}}  \overline v\right)
\]
such that by the  fundamental theorem of calculus we have
\begin{equation}\label{osci}
\begin{aligned}
\Osc(t,\mathcal{L}, v,\overline v) & =   \int_0^t \mathcal{F}(t, \xi, v, \overline{v}) d\xi 
=  t  \mathcal{F}(t, 0, v, \overline{v}) + \int_0^t \int_0^\xi  \partial_s \mathcal{F}(t, s, v, \overline{v} ) d s d \xi\\
& 
=  t   e^{ t \mathcal{L}} f(v,\overline v)+ \int_0^t \int_0^\xi  \partial_s \mathcal{F}(t, s, v, \overline{v} ) d s d \xi.
\end{aligned}
\end{equation}
Next we calculate that 
\begin{multline*}
 \partial_\xi \mathcal{F}(t,\xi,v,\overline v) \\
  = e^{(t-\xi)\mathcal{L}} \Bigl[  
\left.  -   \mathcal{L}\left( f \left(e^{\xi\mathcal{L}} v, e^{ \xi \overline{\mathcal{L}}}  \overline v\right)\right)
+   D_1  f \left(e^{\xi\mathcal{L}} v, e^{ \xi \oD}  \overline v\right)\cdot\mathcal{L}e^{ \xi\mathcal{L}} v 
+  D_{2}  f \left(e^{ \xi\mathcal{L}} v, e^{\xi \oD}  \overline v\right) \cdot  \oD e^{\xi \oD } \overline v
 \right]
\end{multline*}
which yields thanks to our definition \eqref{DlowIntro} that
\begin{equation}
\label{dF} \partial_\xi \mathcal{F}(t,\xi,v,\overline v) =  e^{(t-\xi)\mathcal{L}}\left(  D_{2}  f \left(e^{ \xi\mathcal{L}} v, e^{\xi \oD}  \overline v\right) \cdot 
\mathcal{A} e^{ \xi \overline{ \mathcal{L}}} \overline{v}  + 
 \mathcal{C}[f, \mathcal{L}]  \left(e^{\xi\mathcal{L}} v, e^{ \xi \overline{\mathcal{L}}}  \overline v\right) \right).
 \end{equation}
\begin{rem}
Let us observe that 
\begin{itemize}
\item When  $\mathcal{A}= 0$ (which is the case in the real  \emph{parabolic setting}) or when $f(v,\overline{v})= f(v)$, 
 we   obtain by plugging~\eqref{dF}~into~\eqref{osci} that   
\begin{equation}\label{osciPar}
\begin{aligned}
\|\Osc(t,\mathcal{L}, v,\overline v) -  t e^{t\mathcal{L}} f(v,\overline v) \| & \leq  C t^2 \sup_{s \in [0, t]} \|\partial_s \mathcal{F}(t,s, v, \overline{v}) \|  \\& \leq C t^2
\sup_{\xi \in [0, t]} \left\| \mathcal{C} [f, \mathcal{L}]  \left(e^{\xi\mathcal{L}} v, e^{ \xi \overline{\mathcal{L}}}  \overline v\right) \right\|.
\end{aligned}
\end{equation}
Therefore, by using Assumption \ref{asscom}, we can control the remainder if $v$ is only in $X^{\alpha_{1}}$.
\item  In the    \emph{dispersive setting},  on the other hand,  where $\mathcal{L}= i \mathcal{M}$ with $\mathcal{M}= \overline{ \mathcal{M}}$  the standard Taylor series expansion of the oscillations \eqref{osci}   introduces a classical local error structure involving the full  differential operator $\mathcal{L}$ (similarly to splitting or exponential integrator methods).
Indeed, since we now have  $\mathcal{A}= - 2 \mathcal{L}$, 
we still need that  $v \in X^1 = \mathcal{D}(\mathcal{L})$ to estimate the first term in the right hand side of \eqref{dF}, i.e.,  the term $\mathcal{A} e^{ \xi \overline{ \mathcal{L}}} \overline{v}  $.
\end{itemize}
\end{rem}
%

We shall thus now  develop a first-order approximation which  allows  low regularity approximations also in the dispersive setting $\mathcal{A}= - 2 \mathcal{L}$. 
\subsection{First-order approximation}
To  allow   also  in the dispersive setting for low regularity approximations  we need to tackle   those oscillations of \eqref{osc}   which produce the higher order term in~\eqref{dF}, namely
$$ D_{2}  f \left(e^{ \xi\mathcal{L}} v, e^{\xi \oD}  \overline v\right) \cdot 
\mathcal{A} e^{ \xi \overline{ \mathcal{L}}} \overline{v} .$$
 In order to achieve this, we manipulate the principal oscillations \eqref{osc}  as follows. We write
\begin{equation*}
\begin{aligned}
\Osc(t, \mathcal{L}, v,\overline v) & 
=  \int_0^t e^{ (t-  \xi) \mathcal{L}}  f \left(e^{ \xi\mathcal{L}} v, e^{ \xi \mathcal{L} } \left[e^{\xi \mathcal{A}}  \overline v\right]\right) d\xi,
\end{aligned}
\end{equation*}
recall that $\mathcal{A}= - \mathcal{L} + \oD$, and define the \emph{filtered} function
\begin{equation}\label{filtG}
 \mathcal{N}(t, s,\xi,v,\overline v) = 
e^{(t-  s)  \mathcal{L}}  f \left(e^{ s\mathcal{L}} v, e^{ s  \mathcal{L} }  e^{\xi \mathcal{A}}  \overline v \right).
\end{equation}
Note that the principal oscillations \eqref{osc} can be expressed with the aid of the filter function $\mathcal{N}$ as 
\[
\Osc(t, \mathcal{L}, v,\overline v)  
=   \int_0^t  \mathcal{N}(t, \xi,\xi,v,\overline v) d\xi .
\]
Now the  Taylor series expansion of the filtered function $ \mathcal{N}(t,s, \xi, v, \overline{v})$  around $s=0$ yields that
\begin{equation}\label{osc12}
\begin{aligned}
\Osc(t, \mathcal{L}, v,\overline v)  =  
   \int_0^t  \mathcal{N}(t, 0,\xi,v,\overline v) d\xi +    \int_0^t \int_0^\xi \partial_s  \mathcal{N}(t,s,\xi,v,\overline v) ds d\xi
\end{aligned}
\end{equation}
and we observe that thanks to the filtered structure of $\mathcal{N}$ its derivative $\partial_s  \mathcal{N}$ introduces an  {improved error structure}, namely
\begin{align}\label{dG}
\partial_s  \mathcal{N}(t, s,\xi,v,\overline v) = e^{(t-  s)\mathcal{L}}  \mathcal{C}[f, \mathcal{L}]
 \left(e^{ s\mathcal{L}} v, e^{s\mathcal{L}} e^{\xi\mathcal{A}}  \overline v \right).
\end{align}
The latter can be controlled using Assumption \ref{asscom}.
 This motivates the following expansion of the principal oscillatory integral~\eqref{osci} which builds the basis of our first-order Duhamel integrator \eqref{num1Intro}.
\begin{cor}\label{corOsc1} It holds that
\begin{equation}\label{osc1}
\begin{aligned}
\Osc(t, \mathcal{L}, v,\overline v) 
=e^{t   \mathcal{L}}   \int_0^t   f \left(  v,    e^{\xi \mathcal{A}}  \overline v \right)d \xi + \mathcal{R}_{1,1}(t)
\end{aligned}
\end{equation}
with the remainder
\begin{equation}\label{rem11}
\mathcal{R}_{1,1}(t) = \int_0^t \int_0^\xi 
e^{(t-  s)\mathcal{L}}  \mathcal{C}[f, \mathcal{L}] \left(e^{ s\mathcal{L}} v, e^{s\mathcal{L}} e^{\xi \mathcal{A}}  \overline v \right)d s d\xi
\end{equation}
where $\mathcal{C}[f,\mathcal{L}]$ is defined in \eqref{DlowIntro}.
\end{cor}
\begin{proof}
The assertion just  follows from \eqref{osc12} together with \eqref{dG} noting that $$ \mathcal{N}(t, 0,\xi,v,\overline v) = e^{t   \mathcal{L}}  f \left(  v,    
e^{\xi \mathcal{A}}  \overline v \right).$$
\end{proof}
Thanks to Corollary \ref{corOsc1},  it remains to derive a suitable discretisation of the  integral
\begin{equation}\label{ff}
 \int_0^t  f \left(  v,    e^{\xi \mathcal{A}}  \overline v \right)d \xi
\end{equation}
which can  still have high oscillations, e.g., in the dispersive setting $\mathcal{A} \neq 0$.
\begin{rem} Again, when $f$ is independent of $\overline{v}$  or when  $\mathcal{A}= 0$, 
no additional approximation has to be carried out since
\[
 \int_0^t  f \left(  v,    e^{\xi \mathcal{A}}  \overline v \right)d \xi  = t f(v,\overline v).
\]
On the other hand, when $\mathcal{A} \neq 0$  we need to carefully embed the remaining oscillations $e^{\xi \mathcal{A}}$ in \eqref{ff}  into our numerical discretisation as a simple Taylor series expansion of the latter would produce an error term again involving the full operator $ \mathcal{L}$ since 
\begin{equation}\label{to}
e^{\xi\mathcal{A}} = e^{- 2i \xi \mathcal{L}}  = 1 + \mathcal{O}(\xi \mathcal{L}).
\end{equation}
With an expansion of type \eqref{to} we would in particular come  back to a classical local error  structure,   similar to the one of splitting or exponential integrator methods, without  any improvement in  regularity. 

On the other hand, in the more difficult dispersive setting $\mathcal{A} \neq 0$, the advantage is that $\mathcal{A}$ generates not only a semigroup, but a group
  and hence 
we may go forward and backward in time when approximating the remaining oscillations \eqref{ff}.  The latter is the motivation for our Assumption  \ref{ass1}. 
 We shall also begin to use here our Assumption  on the structure of the nonlinearity~\eqref{nonlin}.
\end{rem}

\begin{lemma}\label{lemf} Under Assumption \ref{ass1} it holds that
\begin{equation}\label{rem12}
\begin{aligned}
\int_0^t  f \left( v,  e^{\xi \mathcal{A}} \overline v \right)d \xi 
=  t  \mathcal{B}(F(v)\cdot  \varphi_1( t \mathcal{A}) G (\overline v)) + \mathcal{R}_{1,2}(t)
\end{aligned}
\end{equation}
with $\varphi_1(z) = \frac{1}{z}(e^z-1)$ and the remainder 
\begin{align}\label{R12z}
\mathcal{R}_{1,2}(t) =  
\int_0^t \int_0^\xi \mathcal{B} \left(  F(v) 
 \cdot  e^{ (\xi - s) \mathcal{A}}  \mathcal{C}[G, \mathcal{A}]
  \left(  e^{ s \mathcal{A}}  \overline v \right) \right)  ds d\xi .
\end{align}
\end{lemma}
\begin{proof}
Thanks  to the structure of the nonlinearity \eqref{nonlin} (in particular using that $ \mathcal{B}$ is linear)  we have
\[
\int_0^t  f \left( v,  e^{\xi \mathcal{A}}  \overline v \right)d \xi 
= \int_0^t \mathcal{B}\left(F(v) \cdot G\left(  e^{\xi \mathcal{A}} \overline v \right)\right)d \xi .
\]
Now we apply the same trick as before to filter out the dominant oscillations in $G$ (triggered by $\mathcal{A}$). For this purpose we set 
\[
\mathcal{N}_2(\xi,\mathcal{A},v) = e^{- \xi \mathcal{A}}G  \left(  e^{\xi \mathcal{A}}  \overline v\right)
\]
(which means since $G(\cdot)$ is a vector   that   $e^{- \xi \mathcal{A}}$ is applied to each component) such that
\[
\mathcal{B}\left(F(v) \cdot G\left(  e^{\xi \mathcal{A}} \overline v \right)\right) = \mathcal{B}\left(F(v) \cdot  e^{ \xi \mathcal{A}}\mathcal{N}_2(\xi,\mathcal{A}, v)\right)  .
\]
In this notation we obtain (using again that $\mathcal{B}$ is linear) that
$$
\begin{aligned}
\int_0^t  f \left( v,  e^{\xi \mathcal{A}} \overline v \right)d \xi 
& =  \int_0^t  \mathcal{B} \left( F(v) \cdot   e^{ \xi \mathcal{A}} \mathcal{N}_2(\xi,\mathcal{A}, v)\right) d\xi\\
& = \int_0^t  \mathcal{B} \left( F(v) \cdot  e^{ \xi \mathcal{A}} \left [\mathcal{N}_2(0,\mathcal{A}, v)+ \int_0^\xi \partial_s \mathcal{N}_2(s,\mathcal{A}, v) ds\right ] \right) d\xi
\\&=  t \mathcal{B} \left( F(v) \cdot  \varphi_1( t \mathcal{A}) G(\overline v)  \right)+ 
 \int_0^t  \int_0^\xi  \mathcal{B} \left( F(v) \cdot  e^{ \xi \mathcal{A}} \partial_s \mathcal{N}_2(s,\mathcal{A},  v) ds \right) d\xi.
\end{aligned}
$$
Using that $\partial_s \mathcal{N}_2(s,\mathcal{A}, v) =   e^{ - s \mathcal{A}}  \mathcal{C}[G, \mathcal{A}]
  \left(  e^{ s \mathcal{A}}  \overline v \right)  $  yields the remainder \eqref{R12z}.
\end{proof}
Collecting the results in Lemma \ref{lem1}, Corollary \ref{corOsc1} and Lemma \ref{lemf} we thus obtain the following first-order {low regularity} approximation to the exact solution $u$ of \eqref{ev}.
\begin{cor}\label{cor11}
The exact solution $u$ of \eqref{ev} can be expanded as 
\[
u(t_n+\tau) = e^{\tau \mathcal{L}} \Big( u(t_n)+  \tau {\red  \mathcal{B}}\left(F (u(t_n)) \cdot  \varphi_1\big(\tau \mathcal{A}\big) G(\overline {u(t_n)})\right)\Big)  + \mathcal{R}_1(\tau,t_n)
\]
where the remainder $\mathcal{R}_1(\tau,t_n)$ is given by  
\begin{equation}\label{R1fin}
\begin{aligned}
 \mathcal{R}_1(\tau,t_n) 
 & =
 \int_0^\tau e^{ (\tau-  \xi) \mathcal{L}}  \Big( f \left(u(t_n+\xi), \overline u(t_n+\xi) \right) 
- f \left(e^{\xi\mathcal{L}} u(t_n), e^{ \xi \oD}  \overline u(t_n)\right)
\Big)
d\xi\\
& +   \red \int_0^t \int_0^\xi 
e^{(t-  s)\mathcal{L}}  \mathcal{C}[f, \mathcal{L}] \left(e^{ s\mathcal{L}} v, e^{s\mathcal{L}} e^{\xi \mathcal{A}}  \overline v \right)d s d\xi  \\
& +   e^{\tau \mathcal{L}} \int_0^\tau \int_0^\xi \mathcal{B} \left( F(u(t_n)) \cdot e^{ (\xi  - s) \mathcal{A}}  \mathcal{C}[ G, \mathcal{A}] \left(  e^{ s \mathcal{A}}  \overline u(t_n) \right) \right) ds d\xi .
\end{aligned}
\end{equation}
\end{cor}

Corollary \ref{cor11}  motivates the first-order scheme defined in \eqref{num1Intro} which
locally yields a low-regularity second-order approximation to the exact solution $u(t)$ of \eqref{ev} at time $t = t_{\ell+1} = \tau \cdot (\ell+1)$. The precise error estimates will be given in Section \ref{locErr1} below.

\begin{rem}
Note that in some cases the approximation of the underlying oscillatory structure of~\eqref{ev} can be even improved. We illustrate this on the example of the KdV equation.
For the KdV equation  the principal oscillatory integral \eqref{osc} takes thanks to the expansion \eqref{osci}   the form
 \begin{equation}
 \begin{aligned}
  \Osc(t,\mathcal{L}, v)  & = t \partial_x v^2  +3   \int_0^t \int_0^\xi 
  e^{s \partial_x^3 } \partial_x^2  \left(e^{-s \partial_x^3} \partial_x v \right)^2 \\&
  =  t \partial_x v^2 + 3 \partial_x  \int_0^t   \Osc(\xi, \mathcal{L}, \partial_x v)d\xi.
  \end{aligned}
 \end{equation}
Taking the derivative with respect to time $t$ of the above relation we observe that
 \begin{align*}
 \Osc(t , \mathcal{L}, \partial_x v) =  \partial_x^{-1}\partial_t   \Osc(t,\mathcal{L}, v)  - v^2.
 \end{align*}
This implies
 \begin{align*}
  \Osc(t , \mathcal{L},  v) & = 
   \frac13 \partial_x^{-1}\partial_t   \Osc(t,\mathcal{L}, \partial_x^{-1}v)  -\frac13 (\partial_x^{-1} v)^2\\
   & = \frac13 e^{-t \partial_x^2}  \left(e^{t \partial_x^2}\partial_x^{-1} v\right)^2 
    -\frac13 (\partial_x^{-1} v)^2.
 \end{align*}
Together with  Lemma \ref{lem1},  we hence recover a   low regularity approximation to the  KdV equation which in the periodic setting collapses exactly to the resonance based scheme proposed in~\cite{HS16}.  \red For optimal second order schemes for periodic KdV we refer to \cite{WuZ,BS20}.\black
 \end{rem}
 
 \subsection{Local error estimates}\label{locErr1}
 \begin{proposition}[Local error estimates]
 \label{proplocal1}
 Let us define 
 \begin{equation}
 \label{deferror}  \mathcal{E}(\tau, t_{n})= u(t_{n+1}) - \Phi^\tau_{\text{num}, 1}(u(t_{n})).
 \end{equation}
 Then we have 
 $$ \mathcal{E}(\tau, t_{n})= \mathcal{E}_{1}(\tau, t_{n}) +   \mathcal{E}_{2}(\tau, t_{n}) +  \mathcal{E}_{3}(\tau, t_{n})$$ 
 with  
 \begin{align}
 \label{E1}
 &  \mathcal{E}_{1}(\tau, t_{n})=\int_0^\tau e^{ (\tau-  \xi) \mathcal{L}}  \Big( f \left(u(t_n+\xi), \overline u(t_n+\xi) \right) 
- f \left(e^{\xi\mathcal{L}} u(t_n), e^{ \xi \oD}  \overline u(t_n)\right)
\Big)
d\xi\\
 \label{E2}&  \mathcal{E}_{2}(\tau, t_{n}) = \red \int_0^t \int_0^\xi 
e^{(t-  s)\mathcal{L}}  \mathcal{C}[f, \mathcal{L}] \left(e^{ s\mathcal{L}} v, e^{s\mathcal{L}} e^{\xi \mathcal{A}}  \overline v \right)d s d\xi \\
\label{E3} &   \mathcal{E}_{3}(\tau, t_{n}) = 
  e^{\tau \mathcal{L}} \int_0^\tau \int_0^\xi \mathcal{B} \left( F(u(t_n)) \cdot e^{(\xi  - s) \mathcal{A}}  \mathcal{C}[ G, \mathcal{A}] \left(  e^{ s \mathcal{A}}  \overline u(t_n) \right) \right) ds d\xi.
 \end{align}
 If   Assumption \ref{ass1}, \ref{assnon}, \ref{asscom}, \ref{asstech} hold  and assuming  that
 \begin{equation}
 \label{alpha1borne} \sup_{[0, T]} \|u(t) \|_{\alpha_{1}} \leq C_{T}
 \end{equation}
 (where $\alpha_{1}$  given by Assumption \ref{asscom}) for some $C_{T}>0,$  then there exists $M_{T}>0$ such that for every $\tau \in (0, 1]$
 $$ \| \mathcal{E}(\tau, t_{n}) \|_{\alpha_{0}} \leq M_{T} \tau^2,  \quad  \| \mathcal{E}(\tau, t_{n}) \|_{a_{0}} \leq M_{T} \tau^{1 + \epsilon},  \quad 0 \leq n \tau \leq T$$
 for some $\epsilon >0$.
 \end{proposition}
 
 \begin{proof}
 Within the proof $M_{T}$ will stand for a harmless number that depends only on $C_{T}$ which changes from line to line.
 We first estimate $\mathcal{E}_{1}(\tau, t_{n}).$ By using Assumption \ref{ass1} and \eqref{tech0}, we get that
 $$ \|\mathcal{E}_{1}(\tau, t_{n}) \|_{\alpha_{0}} \leq M_{T} \tau \sup_{[t_{n}, t_{n}+ \tau]} \| u(t_{n}+ \xi ) - e^{\xi \mathcal{L}} u(t_{n}) \|_{\alpha_{0}}.$$
 Next by using again the Duhamel formula \eqref{duh} and Assumption with the second part  of  \eqref{tech0}, we get that
 $$  \| u(t_{n}+ \xi ) - e^{\xi \mathcal{L}} u(t_{n})\|_{\alpha_{0}}  \leq M_{T} \tau$$
 thanks to \eqref{alpha1borne}.
  This yields
  \begin{equation}
  \label{E1proof}  \|\mathcal{E}_{1}(\tau, t_{n}) \|_{\alpha_{0}} \leq M_{T} \tau^2.
  \end{equation}

 The estimate of  $\mathcal{E}_{2}(\tau, t_{n}) $ directly follows from Assumption \ref{ass1} and Assumption \ref{asscom} (cf. \eqref{C1}). We get thanks to \eqref{alpha1borne} that
 \begin{equation}
 \label{E2proof}  \|\mathcal{E}_{2}(\tau, t_{n}) \|_{\alpha_{0}} \leq \tau^2 M_{T}.
 \end{equation}
 
 In a similar  way, by using  Assumption \ref{asscom}, \eqref{C2} we get that
$$\|\mathcal{E}_{\red 3}(\tau, t_{n}) \|_{\alpha_{0}}  \leq \tau^2 M_{T}.$$

 Consequently, by combining the last estimate and \eqref{E1proof}, \eqref{E2proof}, we get that
 \begin{equation}
 \label{Ealpha0}  \|\mathcal{E}(\tau, t_{n}) \|_{\alpha_{0}} \leq M_{T} \tau^2.
 \end{equation}
 
 It remains to estimate  $ \|\mathcal{E}(\tau, t_{n}) \|_{a_{0}}.$
 We observe that thanks to \eqref{deferror}, the Duhamel formula \eqref{duh} and  \eqref{num1Intro} we can write
 \begin{align*}
  \mathcal{E}(\tau, t_{n})
  & =     \int_{0}^{\tau } e^{(\tau - s )\mathcal{L}}  f(u(t_n+s), \overline{u}(t_n+s))\, ds
  - \tau e^{\tau \mathcal{L}} \mathcal{B} \left( F(u(t_{n}) )\varphi_{1}(\tau \mathcal{A}) G(\overline{u}(t_{n}))\right)\\
&   =  \int_{t_{n}}^{t_{n+1}} e^{(t_{n+1}- s )\mathcal{L}}  f(u(s), \overline{u}(s))\, ds
  - \tau e^{\tau \mathcal{L}} \mathcal{B} \left( F(u(t_{n}) )\varphi_{1}(\tau \mathcal{A}) G(\overline{u}(t_{n}))\right).
\end{align*}
 We then estimate the two terms in the above right hand side  separately in $X^{\alpha_{1}}$ by using \eqref{alpha1borne} 
 and  Assumption \ref{assnon} for the first one (recall  that $\alpha_{1} \in (a_{0}, a_{1}]$)  and \eqref{tech0bis} for the second one.  
 This yields 
 \begin{equation} 
 \label{Ealpha1}\|\mathcal{E}(\tau, t_{n})\|_{\alpha_{1}}
  \leq 
   \tau M_{T}.
   \end{equation}
    To conclude for the estimate of  $ \|\mathcal{E}(\tau, t_{n}) \|_{a_{0}}$, we observe that if $\alpha_{0} \geq a _{0}$, 
    we can just use the continuous  embedding  $X^{\alpha_{0}} \subset X^{a_{0}}$ and hence take $\epsilon= 1.$
     If  $ \alpha_{0}< a_{0}$,  since we always assume that  $\alpha_{1}>a_{0}$, we can use the interpolation properties
      of the $X^\alpha$ spaces. In particular, we obtain that
      $$  \|\mathcal{E}(\tau, t_{n})\|_{a_{0}} \leq  \|\mathcal{E}(\tau, t_{n})\|_{\alpha_{1}}^\theta \|\mathcal{E}(\tau, t_{n})\|_{\alpha_{0}}^{1- \theta}$$
      for some $\theta \in (0, 1).$ Therefore, thanks to \eqref{Ealpha1}, \eqref{Ealpha0}, we obtain that
    $$   \|\mathcal{E}(\tau, t_{n})\|_{a_{0}} \leq M_{T} \tau^{ \theta}  (\tau^2)^{1- \theta} \leq M_{T} \tau^{2 - \theta}$$
    which is the desired estimate since $2- \theta >1$.
    
 \end{proof}
 With the above proposition at hand we can prove our global error estimate in Theorem \ref{theo1}.
 \subsection{Proof of Theorem \ref{theo1}}
 Let us set $e^{n}= u(t_{n}) - u^n$, $u^n$ being the numerical solution computed with \eqref{num1Intro}.
 By using Proposition \ref{proplocal1}, we get that
  $$ e^{n+1} = \Phi_{num, 1}^\tau
(u(t_{n})) - \Phi_{num, 1}^\tau(u^n) + \mathcal{E}(\tau, t_{n}), \quad e^0= 0.$$
 By using the expression of the scheme, we get that
 \begin{equation}
 \label{eqerror1}
 e^{n+1}= 
  e^{\tau \mathcal{L}} \Big[ e^n+  \tau \left(\Psi^\tau_{\text{num,1}}(u(t_{n}) )-  \Psi^\tau_{\text{num,1}}(u^n)\right)
  \Big]+ \mathcal{E}(\tau, t_{n}).
%
 \end{equation}
 By using \eqref{tech1} and Proposition \ref{proplocal1}, we get that for $0 \leq n \tau \leq T$, 
 $$ \|e^{n+1}\|_{\alpha_{0}} \leq \|e^n\|_{\alpha_{0}} + \tau C_{\alpha_{0}}( \|u(t_{n}) \|_{a_{0}}, \|e^n \|_{a_{0}}) \|e^n\|_{\alpha_{0}}
  + M_{T} \tau^2.$$
  In a similar way, by using \eqref{tech2}, and  Proposition \ref{proplocal1}, we obtain that
 $$ \|e^{n+1}\|_{a_{0}} \leq \|e^n\|_{a_{0}} + \tau C_{a_{0}}( \|u(t_{n}) \|_{a_{0}}, \|e^n \|_{a_{0}}) \|e^n\|_{a_{0}}
  + M_{T} \tau^{ 1 + \epsilon}.$$
   We then easily get by induction that for $\tau$ sufficiently small, we have that
   $$ \sup_{0  \leq n \leq N} \|e^n \|_{a_{0}} <+ \infty, \quad  \sup_{0  \leq n \leq N} \|e^n \|_{\alpha_{0}} \leq M_{T} \tau.$$
    This ends the proof. 
  
 \section{Second-order approximation}\label{sec:2}
To establish a second-order \red approximation \black  at low regularity  we need to iterate Duhamel's formula~\eqref{duh}  up to higher order and subsequently embed the underlying (iterated) dominant  oscillations into our discretisation. Surprisingly, it turns out that the main difficulty thereby lies in providing a higher order discretisation of the principal oscillations~\eqref{osc}.

We commence with a Lemma on the second-order iteration of Duhamel's formula.
 \begin{lemma}\label{lem2b}
At second-order we have
\begin{align}\label{u2p}
u(t) = u_2(t) +  \mathcal{R}_{2,0}(t)
\end{align}
with the second-oder iteration of Duhamel's formula given by
\begin{equation}
\begin{aligned}\label{u2pp}
u_2(t) & =e^{ t \mathcal{L}} u_0 + \Osc(t,\mathcal{L}, u_0,\overline u_0)  + \frac{t^2}{2}e^{ t \mathcal{L}} \left( D_1 f \left(  u_0,   \overline  u_0\right)  \cdot  f(u_0,\overline u_0) +  D_2 f\left(   u_0,   \overline  u_0\right)\cdot \overline{f(u_0,\overline u_0)}  \right) ,
\end{aligned}
\end{equation}
where the principal oscillations $ \Osc(t,\mathcal{L}, u_0,\overline u_0) $ are defined in    \eqref{osc} and   the remainder $ \mathcal{R}_{2,0}(t)$ takes the form
\begin{multline}\label{rem20}
 \mathcal{R}_{2,0}(t)  = 
\int_0^te^{(t-  \xi)  \mathcal{L}} \left[  f \left( e^{ \xi \mathcal{L}} u_0 + \xi f(u_0, \overline u_0)  + \mathcal{R}_{2,a}(0,\xi)
,e^{ \xi  \oD}  \overline  u_0 + \xi   \overline{f(u_0,\overline u_0)} + \overline{\mathcal{R}_{2,a}}(0,\xi)  \right) \right. \\  
 \left.   \quad \quad  -
   f \left( e^{ \xi \mathcal{L}} u_0, e^{ \xi  \oD}  \overline  u_0\right)  \right]\, d \xi   \\
   -  e^{ t \mathcal{L}} 
\int_0^t \xi \left(
D_1 f \left(  u_0,   \overline  u_0\right)  \cdot  f(u_0,\overline u_0) +  D_2 f\left(   u_0,   \overline  u_0\right)\cdot \overline{f(u_0,\overline u_0)}
 \right)d\xi  
\end{multline}
with
\begin{equation}\label{rem2}
\mathcal{R}_{2,a}(t_n,\xi) = 
\int_0^\xi e^{(\xi -  s )\mathcal{L}}  f \left(u(t_n+s), \overline u(t_n+s) \right) ds
-\xi f(u(t_n), \overline u(t_n)).
\end{equation}
\end{lemma}
\begin{proof}
First note that by replacing $t$ with $\xi$ in \eqref{duh} we have that
\begin{equation}
\begin{aligned}\label{duh2}
u(\xi) & = e^{ \xi \mathcal{L}} u_0 +  
\int_0^\xi e^{(\xi -  s) \mathcal{L}}  f \left(u(s), \overline u(s) \right) ds  = e^{ \xi \mathcal{L}} u_0 + \xi f(u_0, \overline u_0)  + \mathcal{R}_{2,a}(0,\xi)
\end{aligned}
\end{equation}
with the remainder $ \mathcal{R}_{2,a}(0,\xi)$ given in \eqref{rem2}. The latter  is thereby formally of order $\mathcal{O}\left( \xi^2 \mathcal{L}   f(u,\overline u)\right)$.  The proof then follows by iterating Duhamel's formula, i.e., plugging the expansion \eqref{duh2} into \eqref{duh}, which yields that
\begin{equation*}\label{duh21}
\begin{aligned}
u(t) & = e^{ t \mathcal{L}} u_0 +  
\int_0^te^{(t-  \xi ) \mathcal{L}}  f \left( e^{ \xi \mathcal{L}} u_0 + \xi f(u_0, \overline u_0)  + \mathcal{R}_{2,a}(0,\xi)
,e^{ \xi  \oD}  \overline  u_0 + \xi   \overline{f(u_0,\overline u_0)}  + \overline{\mathcal{R}_{2,a}}(0,\xi) \right) d\red \xi.
\end{aligned}
\end{equation*}
\end{proof}
Thanks to Lemma \ref{lem2b} it thus remains to derive a suitable second-order  approximation to the principal oscillations (cf. \eqref{osc})
\begin{align*}\label{osc2}
\Osc(t,\mathcal{L}, u_0,\overline u_0)
=   \int_0^t e^{(t-  \xi )\mathcal{L}}  f \left(e^{ \xi\mathcal{L}} u_0 ,e^{ \xi  \overline{\mathcal{L}} }\overline u_0\right) d\xi.
\end{align*}
In contrast to the first-order approximation discussed in Section \ref{sec:osc} we now    have to embed iterations of these oscillations  into our discretisation. Following \eqref{filtG} we rewrite the principal oscillations with the aid of
of the filtered function 
\begin{equation}\label{fGf}
\begin{aligned}
 \mathcal{N}(t, s,\xi,v,\overline v) & = 
e^{(t-  s)  \mathcal{L}}  f \left(e^{ s\mathcal{L}} v, e^{ s  \mathcal{L} }  e^{\xi \mathcal{A}}  \overline v \right) =e^{(t-  s)  \mathcal{L}} \mathcal{B}\left( F(e^{ s\mathcal{L}} v) \cdot G\left( e^{ s  \mathcal{L} } e^{\xi \mathcal{A}}  \overline v \right)\right).
\end{aligned}
\end{equation}
With this notation at hand we have that
\begin{align}\label{osc2G}
\Osc(t,\mathcal{L}, v,\overline v)
=    \int_0^t  \mathcal{N}(t,\xi,\xi,v,\overline v) d\xi.
\end{align}
Next we carry out a Taylor series expansion up to second-order of the filtered function $\mathcal{N}(t,s,\xi,v,\overline v)$\red defined in \eqref{fGf} \black around $s=0$. With the the notation $\partial_s \mathcal{N}(t,0,\xi,v,\overline v) =  \partial_s \mathcal{N}(t,s,\xi,v,\overline v) \vert_{s=0} $ this yields that 
\[
 \mathcal{N}(t,s,\xi,v,\overline v) =  \mathcal{N}(t,0,\xi,v,\overline v) + \xi  \partial_s  \mathcal{N}(t,0,\xi,v,\overline v)  + \red \int_0^\xi \int_0^s \black  \partial_{s_1}^2  \mathcal{N}(t,s_1,\xi,v,\overline v)
d{s_1} ds 
\]
where  $ \partial_{s_1}^2  \mathcal{N}(t,s_1,\xi,v,\overline v)$ obeys the improved error structure
\begin{align}\label{2com}
\partial_{s}^2  \mathcal{N}(t,s,\xi,v,\overline v)
= e^{(t-s) \mathcal{L}} \mathcal{C}^2\left[f, \mathcal{L}\right]
 \left(e^{s\mathcal{L}} v, e^{s\mathcal{L}} e^{\xi \mathcal{A}}\overline v \right)
\end{align}
where we recall that 
$$ \mathcal{C}^2\left[f, \mathcal{L}\right] = \mathcal{C}\left[
 \mathcal{C}\left[
f, \mathcal{L}
\right],\mathcal{L}
\right].$$
Together with \eqref{osc2G} we thus obtain 
\begin{equation}\label{osc2G2}
\begin{aligned}
\Osc(t, \mathcal{L}, v,\overline v) & 
=    \int_0^t  \mathcal{N}(t,0,\xi,v,\overline v) d\xi +    \int_0^t  \xi  \partial_s \mathcal{N}(t,0,\xi,v,\overline v)  d\xi 
 \\&   +      \int_0^t \int_0^\xi \int_0^s e^{(t-s_{1}) \mathcal{L}} \mathcal{C}^2\left[f, \mathcal{L}\right]
 \left(e^{s_{1}\mathcal{L}} v, e^{s_{1}\mathcal{L}} e^{\xi \mathcal{A}}\overline v \right)  ds_1 ds d\xi.
\end{aligned}
\end{equation}
Note that in the expansion \eqref{osc2G2} we will have to include the term $\partial_s \mathcal{N}(t,0,\xi,v,\overline v)$ explicitly in our scheme. Thus,   in order to guarante stability of the scheme we next exploit that  formally we have for some $0 \leq\eta\leq t$ that
\begin{equation*}
\partial_s  \mathcal{N}(t,0,\xi,v,\overline v)   =  \frac{\delta_{\zeta }  \mathcal{N}(t,\zeta,\xi,v,\overline v) }{t}\vert_{\zeta = t}+ \mathcal{O}\left(t \partial_{\zeta}^2 \mathcal{N}(t,\eta,\xi,v,\overline v) \right)  
\end{equation*}
with the standard shift operator $$\delta_\zeta g(\zeta) := g(\zeta)- g(0).$$ The latter motivates us to further express the expansion of the oscillations \eqref{osc2G2} as follows
\begin{equation}\label{osc22}
\begin{aligned}
\Osc(t, \mathcal{L}, v,\overline v)  &  =    \int_0^t  \mathcal{N}(t,0,\xi,v,\overline v) d\xi +     \int_0^t \xi \frac{\delta_{\zeta}  \mathcal{N}(t,\zeta,\xi,v,\overline v) }{t}\vert_{\zeta = t}
 d\xi + \mathcal{R}_{ \mathcal{N}}(t)
  \end{aligned}
\end{equation}
(see also \cite{BS20}) where the remainder takes the form
\begin{equation}
\begin{aligned}\label{remG}
\mathcal{R}_{ \mathcal{N}}(t) &=   \int_0^t \xi   \left({\red \partial_s}  \mathcal{N}(t,0,\xi,v,\overline v) -    \frac{\delta_{\zeta}  \mathcal{N}(t,\zeta,\xi,v,\overline v) }{t}\vert_{\zeta = t}
\right) d\xi + 
    \int_0^t \int_0^\xi \int_0^s \partial_{s_1}^2  \mathcal{N}(t,s_1,\xi,v,\overline v)  ds_1 ds d\xi\\
    & = \red -\black  \int_0^t  \xi  \int_0^1 \int_0^{ts} \partial_{s_1}^2 \mathcal{N}(t,s_1,\xi,v,\overline v) ds_1 ds {\red d\xi}
   + 
    \int_0^t \int_0^\xi \int_0^s \partial_{s_1}^2  \mathcal{N}(t,s_1,\xi,v,\overline v)  ds_1 ds d\xi \\
    & =  \red -\black  \int_0^t  \xi  \int_0^1 \int_0^{ts} 
    e^{(t-s_{1}) \mathcal{L}} \mathcal{C}^2\left[f, \mathcal{L}\right]
 \left(e^{s_{1}\mathcal{L}} v, e^{s_{1}\mathcal{L}} e^{\xi \mathcal{A}}\overline v \right) ds_1 ds d \xi 
  \\
  & \quad \quad \quad  +  \int_0^t \int_0^\xi \int_0^s   e^{(t-s_{1}) \mathcal{L}} \mathcal{C}^2\left[f, \mathcal{L}\right]
 \left(e^{s_{1}\mathcal{L}} v, e^{s_{1}\mathcal{L}} e^{\xi \mathcal{A}}\overline v \right) ds_1 ds d\xi.
\end{aligned}
\end{equation}
\begin{rem}\label{par2} Again, by the definition of $\mathcal{N}$ in \eqref{fGf} we obtain for nonlinearities such that  $f(v,\overline v) = f(v)$ that
\[
 \mathcal{N}(t,0,\xi,v,\overline v)   = e^{  t  \mathcal{L}}  f(v,\overline v) \quad\text{and}\quad
 \mathcal{N}(t,t,\xi,v,\overline v)   = f \left(e^{t\mathcal{L}} v, e^{ t  \mathcal{L} }  \overline v\right) \quad \text{for all}\quad \xi.
\]
 This also holds true when  $\mathcal{A}= -\mathcal{L}+\oD = 0$. Thus in these cases
  ($\mathcal{A}= 0$ or $f$ independent of $\overline{v}$)  we do not need to carry out any additional approximation as we simply obtain by   \eqref{osc22} that
\begin{equation*}
\begin{aligned}
\Osc(t, \mathcal{L}, v,\overline v)  &  =  t e^{t   \mathcal{L}}  f(v,\overline v)  +\frac{t}{2}  \delta_t  f \left(e^{t\mathcal{L}} v, e^{ t  \mathcal{L} }  \overline v\right) + \mathcal{R}_{ \mathcal{N}}(t).
  \end{aligned}
\end{equation*}
When  $\mathcal{A} \neq 0$,  we still need to tackle the remaining oscillations in the integral terms of \eqref{osc22}.
\end{rem}
We state a lemma on the integration of the oscillations  in  $\xi  \mathcal{N}(t,s,\xi,v,\overline v)$  which will allow  us to handle the second term in \eqref{osc22} also in the general setting where $\mathcal{A} \neq 0$ and $f$ depends on $\overline{v}$.
\begin{lemma}\label{lem:xiG}
Under Assumption \ref{assnon} we have that
\begin{multline}
\int_0^t \xi \frac{\delta_{\zeta}  \mathcal{N}(t,\zeta,\xi,v,\overline v) }{t}\vert_{\zeta = t} d\xi 
= t \delta_{\xi = t} 
e^{(t-\xi )\mathcal{L}}\mathcal{B} \left( F \left(e^{\xi\mathcal{L}} v\right) \cdot \varphi_2\left(t\mathcal{A}\right)G\left(e^{\xi \mathcal{L}} \overline v\right)\right)  + \mathcal{R}_{d \mathcal{N}_1}(t)\\
\end{multline}
where $\varphi_2(z) = \frac{1}{z}\left(e^z-\varphi_1(z)\right)$ and 
 the remainder takes the form
\begin{equation}\label{remdG}
\begin{aligned}
&  \mathcal{R}_{d \mathcal{N}_1}(t)  =   
\int_0^t \int_0^\xi \frac{\xi}{t}\int_0^t  e^{(t -s_1)\mathcal{L}} 
 \mathcal{C}\left[\Psi_{\xi,s},\mathcal{L}\right]\left( e^{\red s_1\mathcal{L}}  v, e^{s\mathcal{A}}e^{\red s_1\mathcal{L}} \overline v\right)
 ds_1 ds d\xi
  \end{aligned}
\end{equation}
with $\Psi_{\xi,s}( v,w) = 
 \mathcal{B} \left( F\left(  v\right) \cdot e^{(\xi-s) \mathcal{A}}
 \mathcal{C}[G,\mathcal{A}] \left(   w\right)\right)$.
\end{lemma}
\begin{proof} 
Note that
\begin{align*}
\int_0^t \xi \frac{\delta_{\zeta}  \mathcal{N}(t,\zeta,\xi,v,\overline v) }{t}\vert_{\zeta = t} d\xi 
=\frac{1}{t}  \int_0^t \xi  \Big(\mathcal{N}(t,t,\xi,v,\overline v) -  \mathcal{N}(t,0,\xi,v,\overline v)  \Big)d\xi  .
\end{align*}
With Assumption \ref{assnon} at hand we obtain similarly to \eqref{rem12} that
\begin{equation*}
\begin{aligned}
\int_0^t \xi &  \mathcal{N}(t,t,\xi,v,\overline v) d\xi 
 =    \int_0^t \xi \mathcal{B} \left( F\left(e^{t\mathcal{L}} v\right) \cdot G\left( 
e^{t \mathcal{L}} e^{\xi \mathcal{A}} \overline v\right) \right)d \xi\\
&  =    \int_0^t \xi \mathcal{B} \left( F\left(e^{t\mathcal{L}} v\right) \cdot e^{\xi \mathcal{A}}\left[ e^{-\xi \mathcal{A}} G\left( 
e^{t \mathcal{L}} e^{\xi \mathcal{A}} \overline v\right) \right]\right)d \xi\\
&  =    \int_0^t \xi \mathcal{B} \left( F\left(e^{t\mathcal{L}} v\right) \cdot e^{\xi \mathcal{A}}
\left[ 
 G\left( 
e^{t \mathcal{L}}  \overline v\right) + \int_0^\xi \partial_s\left(
e^{-s \mathcal{A}} G\left( 
e^{t \mathcal{L}} e^{s \mathcal{A}} \overline v\right) \right)ds
\right]
\right)d \xi\\
& =  \int_0^t \xi \mathcal{B} \left( F\left(e^{t\mathcal{L}} v\right) \cdot e^{\xi \mathcal{A}}
G(e^{t\mathcal{L}} \overline v) \right)d \xi  + 
 \int_0^t \int_0^\xi \xi \mathcal{B} \left( F\left(e^{t\mathcal{L}} v\right) \cdot e^{(\xi-s) \mathcal{A}}
 \mathcal{C}[G,\mathcal{A}] \left(e^{t \mathcal{L}} e^{s\mathcal{A} }\overline v\right)\right) ds d\xi
\\
& =   t^2 \mathcal{B} \left( F\left(e^{t\mathcal{L}} v\right) \cdot \varphi_2(t \mathcal{A})
G(e^{t\mathcal{L}} \overline v) \right)d \xi  + 
 \int_0^t \int_0^\xi \xi \mathcal{B} \left( F\left(e^{t\mathcal{L}} v\right) \cdot e^{(\xi-s) \mathcal{A}}
 \mathcal{C}[G,\mathcal{A}] \left(e^{t \mathcal{L}} e^{s\mathcal{A} }\overline v\right)\right) ds d\xi.
\end{aligned}
\end{equation*}
In the same manner we obtain 
\begin{equation*}
\begin{aligned}\label{rem22}
\int_0^t \xi &  \mathcal{N}(t,0,\xi,v,\overline v) d\xi =e^{t\mathcal{L}} \int_0^t \xi \mathcal{B} \left( F\left( v\right) \cdot G\left( 
 e^{\xi \mathcal{A}} \overline v\right) \right)d \xi\\
 &  =   t^2e^{t\mathcal{L}}   \mathcal{B} \left( F\left(v\right) \cdot \varphi_2(t \mathcal{A})
G( \overline v) \right)d \xi  + e^{t\mathcal{L}} 
 \int_0^t \int_0^\xi \xi \mathcal{B} \left( F\left(  v\right) \cdot e^{(\xi-s) \mathcal{A}}
 \mathcal{C}[G,\mathcal{A}] \left( e^{s\mathcal{A} }\overline v\right)\right) ds d\xi.
\end{aligned}
\end{equation*}
Using that
\begin{align*}
 & \int_0^t \int_0^\xi \frac{\xi}{t} \Big[ \mathcal{B} \left( F\left(e^{t\mathcal{L}} v\right) \cdot e^{(\xi-s) \mathcal{A}}
 \mathcal{C}[G,\mathcal{A}] \left(e^{t \mathcal{L}} e^{s\mathcal{A} }\overline v\right)\right)
  - e^{t\mathcal{L}} 
 \mathcal{B} \left( F\left(  v\right) \cdot e^{(\xi-s) \mathcal{A}}
 \mathcal{C}[G,\mathcal{A}] \left( e^{s\mathcal{A} }\overline v\right)\right)\Big]ds d\xi
\\&
=  \int_0^t \int_0^\xi \frac{\xi}{t} \int_0^t \partial_{s_1} e^{(t -s_1)\mathcal{L}}  \mathcal{B} \left( F\left(e^{s_1 \mathcal{L}} v\right) \cdot e^{(\xi-s) \mathcal{A}}
 \mathcal{C}[G,\mathcal{A}] \left(e^{s_1 \mathcal{L}} e^{s\mathcal{A} }\overline v\right)\right) d_{s_1} ds d\xi \\
 &= \int_0^t \int_0^\xi \frac{\xi}{t}\int_0^t  e^{(t -s_1)\mathcal{L}} 
 \mathcal{C}\left[\Psi_{\xi,s},\mathcal{L}\right]\left( e^{\red s_1\mathcal{L}}  v, e^{s\mathcal{A}}e^{\red s_1\mathcal{L}} \overline v\right)
 ds_1 ds d\xi
\end{align*}
with $\Psi_{\xi,s}( v,w) = 
 \mathcal{B} \left( F\left(  v\right) \cdot e^{(\xi-s) \mathcal{A}}
 \mathcal{C}[G,\mathcal{A}] \left(   w\right)\right)$ concludes the proof.
\end{proof}
 Finally, to handle the first term  in \eqref{osc22}, we need to derive a suitable second-order approximation to
 $$
 \int_0^t  \mathcal{N}(t,0,\xi,v,\overline v)d\xi = e^{t\mathcal{L}}  \int_0^t \mathcal{B} \left( F\left( v\right) \cdot G\left( 
 e^{\xi \mathcal{A}} \overline v\right) \right) d\xi.
 $$  
 \begin{lemma}\label{lem:G}
 Under Assumption \ref{assnon} we have that
 \begin{equation}\label{rem24}
 \begin{aligned}
\int_0^t   \mathcal{N}(t,0,\xi,v,\overline v) d \xi 
&= t e^{t\mathcal{L}} \mathcal{B}\left(F (v) \cdot \varphi_1\left( t \mathcal{A}\right) G(\overline v) \right)\\&  +  
 t e^{t\mathcal{L}}  \mathcal{B}\left(F (v) \cdot \varphi_2\left( t \mathcal{A}\right) \delta_t \left( e^{- t\mathcal{A}}G \left(  e^{t \mathcal{A} }  \overline v\right) \right)  \right) +  \mathcal{R}_{d\mathcal{N}_2}(t)
\end{aligned}
\end{equation}
where the remainder $ \mathcal{R}_{d\mathcal{N}}(t)$ is given by 
\begin{equation}
\begin{aligned}\label{remF2}
\mathcal{R}_{d\mathcal{N}_2}(t) 
& =
e^{t\mathcal{L}}  \int_0^t \int_0^\xi \int_0^{\xi_1}  \mathcal{B} \left( F\left(v\right) \cdot 
e^{(\xi-\xi_{2})  \mathcal{A}} \mathcal{C}^2\left[
G, \mathcal{A}\right]
\left( e^{\xi_{2}\mathcal{A}} \overline v \right)
\right) d\xi_2 d \xi_1  d\xi\\
& + e^{t \mathcal{L}}  \int_0^t \xi \int_0^1  \int_0^{t s}  \mathcal{B} \left( F\left( v\right) \cdot   
\Big[\
e^{(\xi-s_{1})  \mathcal{A}} \mathcal{C}^2\left[
G, \mathcal{A}
\right]
\left( e^{s_{1}\mathcal{A}} \overline v\right) ds_1 ds\Big]
\right)d\xi.
\end{aligned}
\end{equation}
 \end{lemma}
 \begin{proof}
 We carry out a Taylor series expansion in the spirit of \eqref{rem12}. At higher-order this yields with the notation $   \mathcal{N}_2(t,\mathcal{A},v) =  e^{-t\mathcal{A}}  G\left(  e^{t\mathcal{A}} \overline v\right)$ that
\begin{equation*}
\begin{aligned}
& \int_0^t   \mathcal{N}(t,0,\xi,v,\overline v) d \xi 
 =  e^{t\mathcal{L}}  \int_0^t \mathcal{B} \left( F\left( v\right) \cdot G\left( 
 e^{\xi \mathcal{A}} \overline v\right) \right) d\xi =  e^{t\mathcal{L}}  \int_0^t \mathcal{B} \left( F\left( v\right) \cdot  e^{\xi \mathcal{A}}  \mathcal{N}_2(\xi,\mathcal{A},v)\right) d\xi\\
  & =  e^{t\mathcal{L}}  \int_0^t \mathcal{B} \left( F\left( v\right) \cdot  e^{\xi \mathcal{A}} 
   \left[  
   G(\overline v) + \xi \left[\partial_\xi \mathcal{N}_2({\red \xi} ,\mathcal{A},v)  \right]_{\xi = 0}
+ \int_0^\xi \int_0^{\xi_1} \partial_{\xi_2}^2\mathcal{N}_2({\red \xi_2} ,\mathcal{A},v)   d\xi_2 d \xi_1 
  \right]\right) d\xi.
\end{aligned}
\end{equation*}
Next we use that  \begin{align*}
\partial_\xi \mathcal{N}_2(\xi,\mathcal{A},v) \vert_{\xi = 0}
- \red \frac{1}{t}\delta_\xi \mathcal{N}_2(\xi,\mathcal{A},v)\vert_{\xi = t} \black = 
\red - \black \int_0^1  \int_0^{ts} \partial_{s_1}^2 \mathcal{N}_2\red ( s_1,\mathcal{A},v) \black ds_1 ds
\end{align*}
as well as 
\begin{equation*}
\partial_t^2   \mathcal{N}_2(t,\mathcal{A},v) =  \red \partial_t^2\black\left( e^{-t\mathcal{A}}  G\left(  e^{t\mathcal{A}} \overline v\right) \right)
=  e^{-t\mathcal{A}}  \mathcal{C}^2[G,\mathcal{A}]( e^{t\mathcal{A}} \overline v),
\end{equation*}
where $ \mathcal{C}^2\left[f, \mathcal{L}\right] = \mathcal{C}\left[
 \mathcal{C}\left[
f, \mathcal{L}
\right],\mathcal{L}
\right].$
By the linearity of $\mathcal{B}$ this yields the assertion.
\end{proof}
Note that in the above lemma we again used a finite difference approximation for $\left[\partial_\xi e^{-\xi \mathcal{A}}  G\left( 
 e^{\xi \mathcal{A}} \overline v\right)\right]_{\xi = 0}$ in order to guarantee stability of the scheme. 
 
 Collecting the results in Lemma \ref{lem2b} (in particular \eqref{rem2}), Lemma \ref{lem:xiG} and \ref{lem:G}   yields together with \eqref{osc22} and \eqref{remG} the following  expansion of the exact solution.
\begin{cor}\label{cor:2}
The exact solution $u$ of \eqref{ev} allows with the notation $f^n = f(u(t_n), \overline u(t_n))$ the expansion 
\begin{equation}\label{u2order}
\begin{aligned}
u(t_n+\tau) & = e^{ \tau \mathcal{L}} \red u(t_n)\black+
  \tau e^{\tau \mathcal{L}} \mathcal{B}\left(F (u(t_n)) \cdot \varphi_1\left( \tau \mathcal{A}\right) G(\overline u(t_n)) \right)  \\&+  
 \tau e^{\tau \mathcal{L}}  \mathcal{B}\left(F (u(t_n)) \cdot \varphi_2\left( \tau \mathcal{A}\right) \delta_\tau \left( e^{- \tau\mathcal{A}}G \left(  e^{\tau  \mathcal{A} }  \overline u(t_n)\right) \right)   \right)\\
&
+  \tau \delta_{\xi = \tau} 
e^{(\tau-\xi )\mathcal{L}}\mathcal{B} \left( F \left(e^{\xi\mathcal{L}} u(t_n)\right) \cdot \varphi_2\left(\tau \mathcal{A}\right)G\left(e^{\xi \mathcal{L}} \overline u(t_n)\right)\right)  
  \\& + \frac{\tau^2}{2}e^{ \tau \mathcal{L}} \left(  D_1 f^n \cdot f^n+ D_2 f^n  \cdot \overline{f^n} \right) \\& +  \mathcal{R}_2(\tau,t_n)
\end{aligned}
\end{equation}
where  the remainder $\mathcal{R}_2(\tau,t_n)$ takes  the form
\begin{equation}\label{R2fin}
\begin{aligned}
  \mathcal{R}_2(\tau,t_n)&   = 
\int_0^\tau e^{(\tau -  \xi ) \mathcal{L}}  \left[ f \left( e^{ \xi \mathcal{L}} u(t_n) + \xi f^n + \mathcal{R}_{2,a}(t_n,\xi)
,e^{ \xi  \oD}  \overline  u(t_n) + \xi   \overline{f^n} + \overline{\mathcal{R}_{2,a}}(t_n,\xi) \right)  \right.\\ &
\left. \quad \quad \quad \quad \quad \quad \quad - 
  f \left( e^{ \xi \mathcal{L}} u(t_n), e^{ \xi  \oD}  \overline  u(t_n)\right) \right] d\xi  \\
  & -  e^{ \tau \mathcal{L}} 
\int_0^\tau \xi \left(   D_1 f^n \cdot f^n  +    D_2 f^n \cdot  \overline{f^n} \right)d\xi  \\
& \red +\black   \int_0^\tau  \xi  \int_0^1 \int_0^{\tau s} 
e^{(\tau-s_{1}) \mathcal{L}}
 \mathcal{C}^2\left[
f, \mathcal{L} \right]
\left(e^{s_{1}\mathcal{L}} u(t_{n}), e^{s_{1}\mathcal{L}} e^{\xi \mathcal{A}}\overline u(t_{n}) \right)\red ds_1 ds d\xi \\& 
\red +\black
 \int_0^\tau \int_0^\xi \int_0^s
e^{(\tau-s_{1}) \mathcal{L}}
 \mathcal{C}^2\left[
f, \mathcal{L} \right]
\left(e^{s_{1}\mathcal{L}} u(t_{n}), e^{s_{1}\mathcal{L}} e^{\xi \mathcal{A}}\overline u(t_{n}) \right) ds_1 ds d\xi\\
&+ e^{\tau\mathcal{L}}  \int_0^\tau \int_0^\xi \int_0^{\xi_1}  \mathcal{B} \left( F\left( u(t_n)\right) \cdot 
e^{(\xi-\xi_{2})  \mathcal{A}} \mathcal{C}^2\left[
G, \mathcal{A}\right]
\left( e^{\xi_{2}\mathcal{A}} \overline u(t_{n}) \right)
\right) d\xi_2 d \xi_1  d\xi\\
& + e^{\tau \mathcal{L}}  \int_0^\tau \xi \int_0^1  \int_0^{\tau s}  \mathcal{B} \left( F\left( u(t_{n})\right) \cdot   
\Big[\
e^{(\xi-s_{1})  \mathcal{A}} \mathcal{C}^2\left[
G, \mathcal{A}
\right]
\left( e^{s_{1}\mathcal{A}} \overline u(t_{n})\right) ds_1 ds\Big]
\right)d\xi\\
 &+  \int_0^\tau \int_0^\xi \frac{\xi}{\tau}\int_0^\tau  e^{(\tau -s_1)\mathcal{L}} 
 \mathcal{C}\left[\Psi_{\xi -s},\mathcal{L}\right]\left( e^{\red s_1\mathcal{L}}  u(t_n), e^{s\mathcal{A}}e^{\red s_1\mathcal{L}} \overline u(t_n)\right)
 ds_1 ds d\xi
 \end{aligned}
\end{equation}
with
\begin{equation}\label{Remainder2}
\begin{aligned}
&\mathcal{R}_{2,a}(t_n,\xi) =  
\int_0^\xi e^{(\xi -  s) \mathcal{L}}  f \left(u(t_n+s), \overline u(t_n+s) \right) ds
-\xi f(u(t_n), \overline u(t_n))\\
& \Psi_{\xi}( v,w) = 
 \mathcal{B} \left( F\left(  v\right) \cdot e^{\xi \mathcal{A}}
 \mathcal{C}[G,\mathcal{A}] \left(   w\right)\right), \, \xi \in \mathbb{R}.
\end{aligned}
\end{equation}

\end{cor}

Corollary \ref{cor:2} motivates the second-order scheme  \eqref{num2Intro}
 which locally yields a low-regularity third-order approximation to the exact solution $u(t)$ at time $t = t_{\ell+1} = \tau \cdot (\ell+1)$ . 

\subsection{Local error estimates.}\label{locErr2}
\begin{proposition}[Local error estimates for the second-order scheme]
 \label{proplocal2}
 
 Let 
 \begin{equation}
 \label{defE2} \mathcal{R}_{2}(\tau, t_{n})=  u(t_{n+1}) - \Phi^\tau_{\text{num}, 2}(u(t_{n}))
 \end{equation}
   with $u$ defined in Corollary  \ref{cor:2}. 
    If   Assumption \ref{ass1}, \ref{assnon}, \ref{asscom2}, \ref{asstech2} hold  and assuming  that
 \begin{equation}
 \label{alpha2borne} \sup_{[0, T]} \|u(t) \|_{\alpha_{2}} \leq C_{T}
 \end{equation}
 (where $\alpha_{2}$  is given by Assumption \ref{asscom2}) for some $C_{T}>0,$  then there exists $M_{T}>0$ such that for every $\tau \in (0, 1]$
 \begin{equation}
 \label{O2local}\|\mathcal{R}_{2}(\tau, t_{n}) \|_{\alpha_{0}} \leq M_{T} \tau^3,  \quad  \| \mathcal{R}_{2}(\tau, t_{n}) \|_{a_{0}} \leq M_{T} \tau^{1 + \epsilon},  \quad 0 \leq n \tau \leq T
 \end{equation}
 for some $\epsilon >0$.
 \end{proposition}
 
 \begin{proof}
 Within the proof $M_{T}$ will again stand for a harmless number that depends only on $C_{T}$ which changes from line to line.
 By using Corollary \ref{cor:2}, we can write
 $$ \mathcal{R}_{2}(\tau, t_{n}) = \sum_{j=1}^6   \mathcal{E}_{j}(\tau, t_{n}),$$
 where in the expansion of the right hand side  given by \eqref{R2fin}, we define
 $  \mathcal{E}_{1}(\tau, t_{n})$ as the first 3 lines
  and then each line corresponds to one term.
   The estimates of $\mathcal{E}_{j}$ for $j \geq 2$ is a direct consequence
    of  Assumption \ref{asscom2} and  Assumption \ref{ass1}, we thus get
    $$ \|\mathcal{E}_{j}\|_{\alpha_{0}} \leq M_{T}, \quad j \geq 2.$$
    It thus only remains to estimate $\mathcal{E}_{1}(\tau, t_{n})$.
    By using \eqref{tech02}, we first observe that
   \begin{multline*}
  \mathcal{E}_{1}(\tau, t_{n})=  
  \int_0^\tau e^{(\tau -  \xi ) \mathcal{L}}  \left[ f \left( e^{ \xi \mathcal{L}} u(t_n) + \xi f^n 
,e^{ \xi  \oD}  \overline  u(t_n) + \xi   \overline{f^n}  \right)  - 
  f \left( e^{ \xi \mathcal{L}} u(t_n), e^{ \xi  \oD}  \overline  u(t_n)\right) \right] d\xi  \\
   -  e^{ \tau \mathcal{L}} 
\int_0^\tau \xi \left(   D_1 f^n \cdot f^n  +    D_2 f^n \cdot  \overline{f^n} \right)d\xi
 + \mathcal{E}_{1, 1}(\tau, t_{n}),
   \end{multline*}
   where 
  \begin{equation}
  \label{E11}\| \mathcal{E}_{1, 1}(\tau, t_{n}) \|_{\alpha_{0}} \leq \tau C_{\alpha_{0}}\left( M_{T} ,\, \sup_{\red (t,\xi) \in  [0, T] \times [0,\tau]} \| \mathcal{R}_{2, a}(t\red,\xi\black)\|_{a_{0}}\right)
\sup_{\red (t,\xi) \in  [0, T] \times [0,\tau]} \| \mathcal{R}_{2, a}(t\red,\xi\black)\|_{a_{0}}.
     \end{equation}
   Then, by using only the first estimate of \eqref{picard} for $\alpha = a_{0}$ and the expression \eqref{Remainder2}, we get that
   $$\sup_{\red (t,\xi) \in  [0, T] \times [0,\tau]} \| \mathcal{R}_{2, a}(t\red,\xi\black)\|_{a_{0}}\leq  \tau C_{\alpha_{0}}(\sup_{[0, T]} \|u\|_{\alpha_{2}}) \leq \tau M_{T}.$$ 
    To estimate $ \| \mathcal{R}_{2, a}(t,\red\xi\black)\|_{\alpha_{0}}$, we observe that thanks to the fundamental Theorem of calculus
     and the equation, we have
     \begin{equation}
\begin{aligned}
 & \mathcal{R}_{2, a}(t,\red\xi\black) \\&= \int_{0}^\xi \int_{0}^s  e^{(\xi- s_{1})\mathcal{L}} \Big(
 - \mathcal{L}\left( f(u , \overline{u})\right) + D_{1}f( u, \overline{u})
  \cdot (\mathcal{L}u + f(u, \overline{u})) +  D_{2}f( u, \overline{u})
  \cdot \left(\overline{\mathcal{L}u} + \overline{f}(u, \overline{u}) \right) \Big) \, ds_{1}\, ds\\
  &\quad\red + \xi^2 \varphi_1({\xi\mathcal{L}}) \mathcal{L} f( u(t), \overline{u(t)})
  \end{aligned}
  \end{equation}   
   where in the above integral $u$ and $\overline{u}$ are always evaluated at the time $t+s_{1}.$
   Therefore, by using the assumption \eqref{tech0bisbis2} and \eqref{alpha2borne}, we get that
   $$\sup_{\red (t,\xi) \in  [0, T] \times [0,\tau]} \| \mathcal{R}_{2, a}(t,\red \xi\black)  \|_{\alpha_{0}} \leq \tau^2 M_{T}.$$
   This finally implies thanks to \eqref{E11}
   that
 \begin{equation}
  \|\label{E112}  \mathcal{E}_{1, 1}(\tau, t_{n}) \|_{\alpha_{0}} \leq M_{T} \tau^3.
 \end{equation}
 To estimate the remaining terms in $\mathcal{E}_{1}$, we can write
 \begin{multline*}
 \mathcal{E}_{1}(\tau, t_{n})=  
  \int_0^\tau \xi e^{(\tau -  \xi ) \mathcal{L}}  \int_{0}^1 \left( D_{1}f( e^{\xi \mathcal{L}}u + s \xi f^n, e^{\xi \overline{\mathcal{L}}}
  \overline{u} + s \xi
  \overline{f^n} ) \cdot f^n \right. \\
  \quad \quad \quad \quad \quad \quad\left.+ D_{2}f( e^{\xi \mathcal{L}}u + s \xi f^n, e^{\xi \overline{\mathcal{L}}}\overline{u} + s \xi
  \overline{f^n} ) \cdot \overline{f^n} \right)\, ds \, d\xi  \\
   -  e^{ \tau \mathcal{L}} 
\int_0^\tau \xi \left(   D_1 f^n \cdot f^n  +    D_2 f^n \cdot  \overline{f^n} \right)d\xi
 + \mathcal{E}_{1, 1}(\tau, t_{n})
 \end{multline*}
 where $u$ and $\overline{u}$ are evaluated at $t_{n}$
 and finally
 \begin{multline*}
 \mathcal{E}_{1}(\tau, t_{n})=  
  \int_0^\tau \int_{0}^1 \xi \int_{0}^\xi e^{(\tau - s_{1}) \mathcal{L}}
  \Big[  -\mathcal{L} \left( D_{1}f( V_{s, s_{1}}, \overline{V_{s, s_{1}}})
   \cdot f^n +  D_{2}f( V_{s, s_{1}},  \overline{V_{s, s_{1}}} ) \cdot \overline{f^n}            \right)  \\
   +  D_{1}^2 f( V_{s, s_{1}},  \overline{V_{s, s_{1}}} )\cdot( e^{s_{1} \mathcal{L}} \mathcal{L} u_{n} + s f^n, f^n)
    +D_{12}^2 f( V_{s, s_{1}},  \overline{V_{s, s_{1}}} )\cdot( e^{s_{1} \overline{\mathcal{L}}} \overline{\mathcal{L}}\red\overline{ u_{n}}\black + s 
    \overline{f^n}, f^n) \\+D_{12}^2 f( V_{s, s_{1}},  \overline{V_{s, s_{1}}} )\cdot( e^{s_{1} \mathcal{L}} \mathcal{L} u_{n} + s f^n, 
    \overline{f^n} ) + D_{2}^2 f( V_{s, s_{1}},  \overline{V_{s, s_{1}}} )\cdot (e^{s_{1} \overline{\mathcal{L}}} \overline{\mathcal{L}} \red\overline{ u_{n}}\black + s    \overline{f^n},  \overline{f^n} 
      \Big]  ds_{1} d{s} d{\xi} \\
   +    \mathcal{E}_{1, 1}(\tau, t_{n})
 \end{multline*} 
 where $V_{s, s_{1}}(t_{n})=e^{s_{1} \mathcal{L}}u(t_{n}) + s s_{1}f^n$.
 Consequently, by using \eqref{tech0bisbis3} and \eqref{tech0bisbis2} and \eqref{E112}, we  get that
 $$ \|  \mathcal{E}_{1}(\tau, t_{n}) \|_{\alpha_{0}}
  \leq M_{T} \tau^3$$
  and hence finally the first part of \eqref{O2local}.  
 
 It remains to estimate  $ \|\mathcal{R}_{2}(\tau, t_{n}) \|_{a_{0}}.$
 We use again  \eqref{defE2} and the definition of the numerical flux in \eqref{num2Intro}, to get
 $$ \mathcal{R}_{2}(\tau, t_{n}) = u(t_{n+1}) - e^{\tau \mathcal{L}}u(t_{n})  - \tau e^{\tau \mathcal{L}} \Psi^\tau_{\text{num}, 2}(u(t_n)).$$
 By using the Duhamel formula and the assumption \eqref{tech0bis2}, we easily get that
 $$ \|\mathcal{R}_{2}(\tau, t_{n}) \|_{a_{0}} \leq M_{T} \tau.$$
We finally get \eqref{O2local} from the same interpolation argument as in the study of the first-order case.

 \end{proof}
 
 \subsection{Proof of Theorem \ref{theo2}}
 We follow the same lines as in the convergence proof of our first-order scheme.
Let us set $e^{n}= u(t_{n}) - u^n$, $u^n$ being the numerical solution computed with \eqref{num2Intro}.
  We get that 
  $$
 e^{n+1}= 
  e^{\tau \mathcal{L}} \Big[ e^n+  \tau \left(\Psi^\tau_{\text{num,2}}(u(t_{n}) )-  \Psi^\tau_{\text{num,2}}(u^n)\right)
  \Big]+ \mathcal{R}_{2}(\tau, t_{n}).
%
$$
 By using \eqref{tech12} and Proposition \ref{proplocal2}, we get that for $0 \leq n \tau \leq T$, 
 $$ \|e^{n+1}\|_{\alpha_{0}} \leq \|e^n\|_{\alpha_{0}} + \tau C_{\alpha_{0}}( \|u(t_{n}) \|_{a_{0}}, \|e^n \|_{a_{0}}) \|e^n\|_{\alpha_{0}}
  + M_{T} \tau^3$$
  In a similar way, by using \eqref{tech22}, and  Proposition \ref{proplocal2}, we obtain that
 $$ \|e^{n+1}\|_{a_{0}} \leq \|e^n\|_{a_{0}} + \tau C_{a_{0}}( \|u(t_{n}) \|_{a_{0}}, \|e^n \|_{a_{0}}) \|e^n\|_{a_{0}}
  + M_{T} \tau^{ 1 + \epsilon}.$$
   We then easily get by induction that for $\tau$ sufficiently small, we have that
   $$ \sup_{0  \leq n \leq N} \|e^n \|_{a_{0}} <+ \infty, \quad  \sup_{0  \leq n \leq N} \|e^n \|_{\alpha_{0}} \leq M_{T} \tau^2.$$
    This ends the proof.

\section{Examples}
\label{sectionexamples}
\subsection{Heat, NLS and Ginzburg-Landau type equations}
To illustrate our general theory, we  shall consider  the  equation
\begin{equation}
\label{GL}\partial_t u  -  D \Delta u = \gamma  u (1 - |u|^2), \quad u_{/t=0}= u_{0}
\end{equation}
where $u \in \mathbb{C}$, and $D, \, \gamma \in \mathcal{C}$ with $\text{Re}\, D \geq 0$
 so that  nonlinear Schrödinger (NLS) type equations and nonlinear heat equations are included in the same model.
  We assume that the problem is set for $x \in \Omega \subset \mathbb{R}^d$, $d \leq 3$,  where $\Omega$ is a smooth open set with compact
  boundary and we add the Dirichlet boundary condition
  \begin{equation}
  \label{dir}
  u(t, \cdot)_{/\partial \Omega}= 0.
  \end{equation}
We thus  set $X= L^2(\Omega)$ and $\mathcal{L} = D \Delta$ with domain $D(\mathcal{L})= H^2 \cap H^1_{0}.$
 By standard characterization, we have that $X^\alpha = H^{2 \alpha} \cap H^1_{0}$ for $1/2 \leq \alpha \leq 1$
 while for $0 \leq \alpha <1/2,$  $X^\alpha = H^{2 \alpha}.$ 
 
 Assumption \ref{ass1} is then matched. 
  The ``nonlinearity" is given by 
  $$ f(u, \overline{u}) = \gamma u  - \gamma u^2 \overline{u}$$
   so that $\mathcal{B}= Id$.
   
   We shall check the assumptions for first-order convergence,  we fix  $a_{0}$ any number strictly bigger than $d/4$ arbitrarily close to it 
    so that $\red H^{2 a_{0}}(\Omega)\subset  L^\infty$ and $a_{1}= 1$ (the propagation of  higher order regularity
    would require additional compatibility conditions for the initial data).

   The estimate \eqref{picard} then follows by standard Moser-Gagliardo-Nirenberg-Sobolev estimates since \red for every $\alpha \geq 0$\black
    $$ \||u|^2 u \|_{H^{2 \alpha}}  \lesssim \|u \|_{H^{2a_{0}}}^2 \|u \|_{H^{2 \alpha}}, \quad 2a_{0} >d/2.$$
    
    Let us consider the scheme \eqref{num1Intro} for this model and measure the error in $L^2$ so that
     $\alpha_{0}= 0$.
     
     Then, we have the explicit expression 
     \begin{equation}\label{cNLS}
     \mathcal{C}[f, \mathcal{L}](v,w)=\gamma D \left(  \Delta ( v^2 w ) - 2  v  \Delta v  w - v^2 \Delta w\right)
      = \gamma D \left( 4 v \nabla v \cdot \nabla w  + 2 \nabla v \cdot \nabla v w \right).\end{equation}
       We thus deduce that
    $$ \|\mathcal{C}[f, \mathcal{L}](v,w)\|_{L^2} \lesssim  ( \|v\|_{L^\infty}+  \|w\|_{L^\infty}) (   \|\nabla v \|_{L^4}
     \|\nabla w \|_{L^4} + \|\nabla v \|_{L^4}^2).$$
      Since $ \red H^{2a_{0}} \subset L^\infty $, and $\red H^{s} \subset  L^4$ with $s= d/4$, we can take
       $$\alpha_{1}=  { 1 \over 2} (1+ {d \over 4})$$
     (note that we can always assume that $d/2 <2a_{0} < 1 + d/4$ when $d \leq 3$.)
     
     Next, we have that 
     $  F(u) \cdot e^{s \mathcal{A}} \mathcal{C}[G, \mathcal{A}](w)= 0$, since $f$ is linear in $\overline{u}$.
      Therefore \eqref{C2} is also matched.
      
      It remains to check Assumption \eqref{asstech} (recall that we take $\alpha_{0}= 0$ here).
      \begin{itemize}
      \item The first part of \eqref{tech0} is again a consequence of $ \red H^{2 a_{0}}(\Omega) \subset  L^\infty $ 
       For the second part, we use
       $$ \| |u|^2 u \|_{L^2} \leq \|u \|_{L^\infty} \| u \|_{L^4}^2$$
       and the estimate follows from the embedding  $\red  H^{d \over 4} \subset  L^4.$ 
       \item In this specific case, we have
         $$\mathcal{B}( (F(u)) \cdot \varphi_{1}(\tau, \mathcal{A}) G (w) )= 
          \gamma u - \gamma u^2  \varphi_{1}(\tau, \mathcal{A}) w$$
         since $ \varphi_{1}(\tau, \mathcal{A})$ is bounded on Sobolev spaces, the estimates \eqref{tech0bis}, \eqref{tech1}
         and \eqref{tech2} then  also follow from the above arguments.
      \end{itemize}
      
      We then obtain from Theorem \ref{theo1},  that:
      
      \begin{cor}\label{cori}
      For $1 \leq d \leq 3$ and 
       for every $u_{0}  \in H^{ 1 + {d \over 4}}(\Omega) \cap H^1_{0}(\Omega)$, 
       let $T>0$ and $u \in \mathcal{C}([0, T], H^{ 1 + {d \over 4}}(\Omega) \cap H^1_{0}(\Omega) )$
  the unique solution of \eqref{GL} given by Theorem \ref{localwp}, let  $u^n$ denote the numerical solution given by \eqref{num1Intro}.
   Then there exists $C_{T}>0$ such that
 $$ \|u(n \tau) - u^n \|_{L^2} \leq C_{T} \tau, \quad 0 \leq n \tau \leq T.$$
          \end{cor}
          Note that we require only $1+ d/4$ derivatives to get first-order convergence in contrast to classical schemes which require (at least) $2$ derivatives.
           The same result holds for periodic boundary conditions.
           Corollary \ref{cori} does not use any smoothing properties of the PDE, it could be improved by using parabolic smoothing
            if $\mbox{Re}\,  {\red D} >0$. It could also be improved when $\mbox{Re}\, {\red D} = 0$  if $\Omega$ is nice (non-trapping exterior domain), or if $\Omega= \mathbb{T}^d$ by using tools from dispersive PDE in the convergence analysis of the scheme (e.g. discrete Bourgain-space type estimates) see \cite{ORS19}, \cite{ORS20}.
            
            Also note that for periodic boundary conditions  ($x \in \mathbb{T}^d$) or the full space ($x \in \mathbb{R}^d$)
             if we measure the error  in high order Sobolev spaces $H^{r}$ with $r=2 \alpha_{0}>d/2$ we obtain the global error estimate $ \|u(n \tau) - u^n \|_{r} \leq C_{T} \tau$ for solutions in $H^{r+1}$.

            \begin{rem}
            \red In the special case of the periodic Schr\"odinger equation with polynomial nonlinearities $u^p \overline u^m$ the second-order Duhamel integrator \eqref{num2Intro} improves previous  methods (\cite{KOS19}). For periodic boundary conditions it is similar to the second-order scheme developed in  \cite{BS20} based on the combination of a resonance based discretisation (which is so far restricted to periodic boundary conditions) coupled with suitable filter functions. For the favorable error behaviour  at low regularity, see in particular  \cite[Fig. 3]{BS20}.
            \end{rem}
            \begin{rem}\label{rem:splitNLS}
           In case of the  cubic Schr\"odinger equation  the filtered Lie splitting \eqref{numFLieIntroe} takes the form
\begin{equation}\label{numFLieIntro}
u^{\ell+1}=  e^{\tau \mathcal{L}} \e^{  \tau u^\ell \left(\Psi(\tau)\overline {u^\ell}\right) } u^\ell\quad \text{with} \quad  \Psi(\tau) = \varphi_1\big(- 2i \tau \Delta \big).
\end{equation}
In contrast to its unfiltered counterpart ($\Psi(\tau) \equiv 1$), the filtered Lie splitting \eqref{numFLieIntro} only \red involves  one additional derivative of the solution in the local error (i.e, $\nabla u$) instead of two (i.e., $\Delta u$). This can be seen by reformulating the filtered Lie splitting as follows
\begin{align*}
u^{\ell+1}=  e^{\tau \mathcal{L}} \left[\left (1+ \tau u^\ell  \left(\Psi(\tau)\overline {u^\ell}\right)  \right) u^\ell\right]
+\mathcal{R}(\tau,u^\ell)
\end{align*}
where the first term corresponds to the low regularity scheme \eqref{num1Intro}, which only involves first order derivatives in the local error, see \eqref{cNLS}, and the remainder $$\mathcal{R}(\tau,u^\ell) =e^{\tau \mathcal{L}} \left[\left ( \e^{  \tau u^\ell \left(\Psi(\tau)\overline {u^\ell}\right) }
- 1 - \tau u^\ell  \left(\Psi(\tau)\overline {u^\ell}\right)  \right) u^\ell\right]$$ is of order $\mathcal{O}(\tau^2 q(u^\ell))$ for some polynomial $q$.
            \end{rem}

            \subsection{Half-waves}\label{secHW}
            Let us consider the half-wave equation
   \begin{equation} i \partial_t u + \vert \nabla \vert u =  \pm |u|^2 u, \, \mathcal{L}= i \vert \nabla \vert,  \quad u_{/t=0}= u_{0}.
   \label{half}
   \end{equation}
  \red In this situation the  low regularity scheme at first-order takes the form (cf.  \eqref{num1Intro})
   \begin{equation*}
u^{\ell+1}=    e^{i \tau  \vert \nabla \vert} \Big( u^\ell \mp i   \tau (u^\ell)^2 \varphi_1\big(-2i \tau  \vert \nabla \vert \big)  \overline {u^\ell} \Big) \end{equation*}
and at second-order it reads (cf. \eqref{num2Intro})
\begin{equation*}\begin{aligned}
u^{\ell+1}& = e^{i \tau  \vert \nabla \vert} u^\ell \mp i   \tau e^{i \tau \vert \nabla \vert}  \left( (u^\ell)^2  \varphi_1\big(-2i \tau\vert \nabla \vert \big)   \overline u^\ell\right)  
\mp i   \tau 
  \left(e^{i \tau  \vert \nabla \vert} u^\ell\right)^2 \varphi_2\big(-2i \tau  \vert \nabla \vert \big) e^{i \tau  \vert \nabla \vert} \overline u^\ell
\\&\pm i   \tau 
e^{i \tau   \vert \nabla \vert} \left(  \left(  u^\ell\right)^2 \varphi_2\big(-2i \tau  \vert \nabla \vert \big) \overline u^\ell\right) 
 - \frac{\tau^2}{2}e^{i \tau  \vert \nabla \vert}  \left(     \vert u^\ell \vert^4  u^\ell  \right).
  \end{aligned}
\end{equation*}
  In the following we analyze the order of convergence of the above schemes. \black
  \subsubsection{First-order scheme}
  We shall first analyze our first-order scheme in dimension $1$  with periodic boundary conditions thus  for $x \in \mathbb{T}$, $d=1$.
   Here, we have
   $$ f(u, \overline{u})= \red \mp i \black  u^2 \overline{u}$$
   so that again, $\mathcal{B}= Id$.
             
        We   set $X= L^2(\mathbb{T}),$  $\mathcal{L} =  i |\nabla |$ with domain $D(\mathcal{L})= H^1(\mathbb{T}).$
 We thus  have that $X^\alpha = H^{ \alpha}(\mathbb{T})$.
 Assumption \ref{ass1} is then matched.    
    We fix  now  $a_{0}$ any number strictly bigger than $1/2$ arbitrarily close to it, 
    and take $a_{1}= \alpha_{1}= 3/4$ and we will again measure the error in $L^2$ so that $\alpha_{0}=0$.
    Assumptions \ref{assnon}, \ref{asstech}  are matched thanks to the same arguments as above.
     Then, we have that
     $$ \mathcal{C}[f, \mathcal{L}](v,w)=\pm i  \left(  |\nabla| ( v^2 w ) - | \nabla | (v^2)   w - v^2 |\nabla | w  + (|\nabla |( v^2) - 2 v |\nabla v| )w\right).$$
      To estimate this term, we shall use the following estimate
      $$ \| |\nabla | (f g) - |\nabla| f g - f |\nabla |g  \|_{L^2} \lesssim \| f \|_{H^{ {1 \over 2 } + { d \over 4 }}} \| g \|_{H^{ {1 \over 2 } + { d \over 4 }}} $$
     which can be obtained as a consequence of some recent version of the generalized Leibnitz rule
     \red \cite{Dong-Li}\black. We shall provide a direct  proof of the above estimate which is valid for the torus in  Section \ref{sectiontech}. The latter in particular implies that 
   \begin{multline*} \| \mathcal{C}[f, \mathcal{L}](v,w) \|_{L^2}
     \lesssim \|v^2 \|_{H^{ {1 \over 2}+ {d \over 4} }} \|  w \|_{H^{{1 \over 2}+ {d \over 4}}}
      + \|w\|_{L^\infty} \|\nabla |( v^2) - 2 v |\nabla v|\|_{L^2}\\
       \lesssim \|v\|_{L^\infty} \|v \|_{H^{ {1\over 2} + {d \over 4}}}
   \|w \|_{H^{ {1\over 2} + {d \over 4}}} + \|w \|_{L^\infty} \|v \|_{H^{ {1\over 2} + {d \over 4}}}^2.
   \end{multline*}
    Therefore, with   $\alpha_{1}= 3/4$, we get
    $$ | \mathcal{C}[f, \mathcal{L}](v,w) \|_{L^2}
     \lesssim  \|w\|_{H^{\alpha_{1}}} \|v\|_{H^{\alpha_{1}}}^2.$$

     Next, we have that 
     $  F(u) \cdot e^{s \mathcal{A}} \mathcal{C}[G, \mathcal{A}](w)= 0$, since $f$ is linear in $\overline{u}$
      therefore \eqref{C2} is also matched.
      We can check \eqref{picard} and Assumption \ref{asstech} as in the previous example.
      
          We then obtain from Theorem \ref{theo1},  that
      
      \begin{cor}
      For 
       for every $u_{0}  \in H^{3 \over 4}({\mathbb{T}})$. 
       Let $T>0$ and $u \in \mathcal{C}([0, T], H^{3 \over 4}(\mathbb{T})$
  the unique solution of \eqref{GL} given by Theorem \ref{localwp}, let  $u^n$ denote the numerical solution given by \eqref{num1Intro}.
   Then there exists $C_{T}>0$ such that
 $$ \|u(n \tau) - u^n \|_{L^2} \leq C_{T} \tau, \quad 0 \leq n \tau \leq T.$$
          \end{cor}
Note that classical schemes would require initial data (at least) in $H^1$.

In higher dimensions,  $2$, $3$ for example, at first-order our convergence result  does not improve on the classical regularity assumption.
 The main reason is that the order of the operator $\mathcal{L}$, which is $1$  in this case,  is smaller than $d/2$ which is  the minimal
  $a_{0}$ that we can take to get a local Cauchy Theory without using more sophisticated harmonic analysis tools.
   In this case,  the estimate for the commutator  $\mathcal{C}[f, \mathcal{L}]$ is not better than  the estimate of each
    of the terms alone.
 We could improve the first-order error estimate  by adapting the harmonic analysis  tools developed in \cite{ORS20}. Nevertheless, even without any sophisticated tools our second-order convergence result yields a significant improvement as we shall see below.
   
   \subsubsection{Second-order scheme}
   To illustrate our second-order theory, we shall now consider \eqref{half} for $d = 2$ still with periodic boundary conditions
   i.e $x \in \mathbb{T}^2.$ We still have $X^\alpha = H^\alpha(\mathbb{T}^d)$. We choose
   again to measure the error in $L^2$ so $\alpha_{0}= 0$ and we fix $a_{0}>d/2$ arbitrarily close to it.
    We first check the estimate \eqref{C12} to fix $\alpha_{2}$. Note that again \eqref{C22} and \eqref{C32} are trivially satisfied, 
    $\mathcal{C}[G, \mathcal{A}]= 0$ since $G$ is linear.
    Let us set $T(v_{1}, v_{2}, v_{3})= v_{1}v_{2}v_{3}.$
     We shall estimate $ \|\mathcal{C}^2[T, \mathcal{L}](v_{1}, v_{2}, v_{3})\|_{L^2}$. The estimate of
     $ \|\mathcal{C}^2[f, \mathcal{L}](v, w)\|_{L^2}$ then follows by observing that
      $ f(v, w)= T(v,v, w).$
       By expanding $\mathcal{C}^2[T, \mathcal{L}](v_{1}, v_{2}, v_{3})$, we get that
  \begin{equation}
  \label{comdemi}\mathcal{C}^2[T, \mathcal{L}](v_{1}, v_{2}, v_{3})= |\nabla |^2(v_{1} v_{2} v_{3}) -
   |\nabla|^2 v_{1} v_{2} v_{3} - v_{1}| \nabla |^2 v_{2} v_{3 }- v_{1} v_{2} | \nabla |^2 v_{3}+  \mathcal{R}_{1}+ \mathcal{R}_{2}
   \end{equation}
   where 
   $$\mathcal{R}_{1} = \sum_{i,j,k} v_{i}  |\nabla |v_{j}\, |\nabla | v_{k}, \quad \mathcal{R}_{2}= \sum_{i,j,k}\left( |\nabla | ( v_{i} v_{j}  |\nabla| v_{k}) - v_{i}v_{j}|\nabla|^2v_{k}\right).$$  
    The sums are finite sums, we do not need to explicit them.
   We can easily estimate $\mathcal{R}_{1}$ by using the H\"older inequality and the Sobolev embedding $\red   H^{1 \over 2}\subset  L^4 $
    in dimension $2$. We get
    $$ \|\mathcal{R}_{1}\|_{L^2} \lesssim   \sum_{i,j,k} \|v_{i}\|_{L^\infty} \| |\nabla| v_{j}\|_{L^4} \||\nabla |v_{k}\|_{L^4}
     \lesssim \sum_{i,j,k} \|v_{i}\|_{H^{a_{0}}} \|v_{j }\|_{H^{3 \over 2}} \|v_{k}\|_{H^{3 \over 2}}.$$
      For $\mathcal{R}_{2}$ we use the Kato-Ponce commutator estimate (see for example  \cite{Dong-Li}  Theorem 1.9) to get
      $$ \| \mathcal{R}_{2} \|_{L^2} \lesssim \sum_{i,j,k} \| \nabla ( v_{i} v_{j}) \|_{L^4} \| |\nabla| v_{k}\|_{L^4}
       \lesssim   \sum_{i,j,k} \|v_{i}\|_{L^\infty}   \|v_{j }\|_{H^{3 \over 2}} \|v_{k}\|_{H^{3 \over 2}}.$$ 
      It remains to estimate the leading term in \eqref{comdemi}. Since $|\nabla|^2=- \Delta $ we can expand
       with the classical Leibnitz formula. We obtain
     $$ \mathcal{C}^2[T, \mathcal{L}](v_{1}, v_{2}, v_{3})= - \sum_{i,j,k} v_{i} \nabla  v_{j } \cdot \nabla  v_{k} 
      + \mathcal{R}_{1} + \mathcal{R}_{2}.$$
      Therefore, from the same estimates as previously, we obtain that
      $$ \| \mathcal{C}^2[T, \mathcal{L}](v_{1}, v_{2}, v_{3}) \| \lesssim  \|v_{1 }\|_{H^{3 \over 2}}   \|v_{2 }\|_{H^{3 \over 2}} \|v_{3}\|_{H^{3 \over 2}}.$$
       We thus choose $\alpha_{2}= {3 \over2 }.$
        The other assumptions are then again a consequence of the tame estimate, for $s \geq 0$:
        $$ \| u v w \|_{H^s} \lesssim \|u \|_{H^{a_{0}}} \|v\|_{H^{a_{0}}} \|w\|_{H^s} + \|v \|_{H^{a_{0}}} \|w\|_{H^{a_{0}}} \|u\|_{H^s}
         +  \|w \|_{H^{a_{0}}} \|u\|_{H^{a_{0}}} \|v\|_{H^s}.$$
      
      We thus obtain:
    \begin{cor}
      For 
        every $u_{0}  \in H^{3 \over 2}({\mathbb{T}^2})$. 
       Let $T>0$ and $u \in \mathcal{C}([0, T], H^{3 \over 2}(\mathbb{T}^2))$
  the unique solution of \eqref{half} given by Theorem \ref{localwp}, let  $u^n$ denote the numerical solution given by \eqref{num2Intro}.
   Then there exists $C_{T}>0$ such that
 $$ \|u(n \tau) - u^n \|_{L^2} \leq C_{T} \tau^2, \quad 0 \leq n \tau \leq T.$$
          \end{cor} 
          Note that   second-order convergence of  classical  schemes for the Half-wave  equation~\eqref{half}   requires (at least) 
  $H^2$ solutions.

   \subsection{Wave and Klein-Gordon type equations}

Our  general  framework can also be applied to Klein--Gordon and wave equations: Let us for instance consider the nonlinear Klein--Gordon equation
\begin{equation}\label{kgrO}
\partial_{tt} z - \Delta z  + m^2 z =  g(z), \quad z(0) = u_{0}, \quad \partial_{t} z(0) =  u_{1}, 
\end{equation}
where for simplicity we assume non zero mass $m \neq 0$ and  real-valued solutions  $z(t,x)\in \mathbb{R}$. However, we may  also deal with the complex setting $z(t,x)\in \mathbb{C}$ and wave equations $m= 0$. 

In a first step let us rewrite the Klein--Gordon equation \eqref{kgrO} as a first-order complex system. We consider 
that $x \in \mathbb{T}^d$ for $d \leq 3$ so that 
\[
 \cnab = \sqrt{-\Delta +m^2}
 \]
 can be defined with Fourier series. We get 
\begin{equation}\label{kgroo}
i \partial_t u = - \cnab u +  \cnab^{-1}g\big((\textstyle \frac12 (u+ \overline u)\big).
\end{equation}
The latter is obtained via the classical transformation
 \begin{equation}
 \label{uz}
u= z - i \cnab^{-1} \partial_t z,
\end{equation}
where
\[
 \cnab = \sqrt{-\Delta +m^2}
 \] is  an  invertible  operator (recall that $m \neq 0$) and $$z = \frac12 (u+\overline u) = \mathrm{Re}(u).$$  
 
 We thus define $\mathcal{L}= i  \cnab$, with $X= L^2(\mathbb{T}^d)$, $D(\mathcal{L})= H^1(\mathbb{T}^d)$ and hence
  $X^\alpha= H^\alpha(\mathbb{T}^d)$.
   Assumption \ref{ass1} is clearly satisfied.
 
To illustrate our theory,  we shall consider two classical nonlinearities:
 \begin{itemize}
 \item  power type:  $g(z)= z^2$
 \item Sine-Gordon: $ g(z)= -  \sin z.$
 \end{itemize}
 We have  
$$f(u, \overline{u})=  - i \cnab^{-1}g\big((\textstyle \frac12 (u+ \overline u)\big)$$
 so that $f$ is under the form \eqref{nonlin} with
 \begin{align}
 \label{cas1} & f(v,w)=  \mathcal{B}( F(v) \cdot G(w))=  - \red\tfrac14\black i \cnab^{-1} \left(  v^2 + w^2 + 2 v w \right), \quad g(z) = z^2, \\
 \label{cas2} &  f(v,w) =   \mathcal{B}( F(v) \cdot G(w))=   i \cnab^{-1} \left(\mathrm{sin}(\textstyle \frac12 v) \mathrm{cos}(\textstyle \frac12 w) + \mathrm{cos}(\textstyle \frac12 v) \mathrm{sin}(\textstyle \frac12 w)\right), \quad g(z) = - \sin z.
 \end{align}
 Note that in this case, $\mathcal{B}$ is non-trivial, we have, $ \mathcal{B}=\cnab^{-1} .$
  We can thus use our schemes for this formulation.  
 A natural space to measure the error for wave type equations in terms of $(z, \partial_{t}z)$ is $H^1 \times L^2.$
  In terms of $u$ with the definition \eqref{uz}  the $H^1$ norm is hence a natural choice. 
   We shall thus choose  $\alpha_{0}= 1$.
  \subsubsection{First-order scheme} 
  We shall study our first-order-scheme \eqref{num1Intro}, in dimension $1$, $d=1$.
   Let us then check our assumptions \ref{assnon}, \ref{asscom}, \ref{asstech} with $\alpha_{0}= 1$.
  
    Let us start with the nonlinearity given by  \eqref{cas1}.
    We can take  $a_{0}= 0$ and we 
      we can take $a_{1}= \alpha_{1} = {1 \over 2} + {1 \over 4}= {3 \over 4}$.
       Indeed, let us first check Assumption \ref{asscom}. 
        We use again  the following inequality
    \begin{multline}\label{eik}
      \| \cnab^{-1}( \cnab (v w) - \cnab v  w - v \cnab w) \|_{H^1} \\ \lesssim  \|\cnab (v w) - \cnab v  w - v \cnab w) \|_{L^2}
   \lesssim  \|v\|_{H^{\alpha_{1}}} \|w\|_{H^{\alpha_{1}}}
      \end{multline}
      where we have used Lemma \ref{lemtech} (see below)  for the last inequality  together 
     with the Sobolev embedding $ H^{d \over 4} \subset  L^4$. 
       The  estimate \eqref{eik} implies \eqref{C1}, \eqref{C2} and motivates the choice of $\alpha_{1}.$
        Let us now check the other assumptions.
       We have \red by duality (since $H^1 \subset L^\infty$) that\black
     $$ \|  \cnab^{-1} \left( v  w \right) \|_{L^2} \lesssim  \|v w \|_{H^{- 1}} \lesssim  \|v w\|_{L^1} \lesssim \|v\|_{L^2} \|w\|_{L^2}$$
      and  
      $$  \| f(u, \overline{u}) \|_{H^{3\over 4 }} \lesssim \| |u|^2 \|_{H^{ - {1 \over 4}}} \lesssim \| |u|^2 \|_{L^2} \lesssim \|u \|_{L^2} \|u\|_{L^\infty}
       \lesssim \|u\|_{L^2} \|u \|_{H^{3\over 4}}.$$ 
       This yields \eqref{picard} (the case $\alpha \in (0, 3/4]$ can be obtained from similar arguments and is actually not needed)
        and \eqref{tech2}.
     The estimates \eqref{tech1},  \eqref{tech0}, \eqref{tech0bis} are a consequence of the following  estimate:  for $ \alpha \leq 1$
    $$ \| \cnab^{-1} \left( v  w \right) \|_{H^\alpha } \lesssim \|v w \|_{L^2} \lesssim \|v \|_{L^2} \| w \|_{L^\infty} \lesssim \|v\|_{L^2} \|w \|_{H^{3\over 4}}.$$

 The  nonlinearity \eqref{cas2} can be handled in a similar way,   the only  new non-trivial estimates to obtain are
  \eqref{C1}, \eqref{C2}. We first observe that 
 $$  \| \cnab^{-1}( \cnab \sin u  -  \cos u \cnab u)) \|_{H^1} \\ \lesssim  \|\cnab \sin u  - \cos u \cnab u \|_{L^2}.$$
  So that we only need to estimate $ \|\cnab \sin u  - \cos u \cnab u \|_{L^2}$.
   Details will be given below in Lemma~\ref{lemtech2}.
    We obtain again that
  $$ \|\cnab \sin u  - \cos u \cnab u \|_{L^2} \lesssim C( \|u\|_{H^{3 \over 4}}).$$
  
  We thus obtain that
  
\begin{cor}
      For 
       for every $u_{0}  \in H^{3 \over 4}({\mathbb{T}})$, $u_{1} \in H^{- {1 \over 4}}(\mathbb{T})$ 
       Let $T>0$ and $z \in \mathcal{C}([0, T], H^{3 \over 4}(\mathbb{T})) \cap \mathcal{C}^1([0, T], H^{- {1 \over 4}}(\mathbb{T})) $
  the unique solution of \eqref{kgrO}, with $g(z)= z^2$ or $g(z)= \sin z$. Let  $u^n$ denote the numerical solution given by \eqref{num1Intro}.
   Then there exists $C_{T}>0$ such that
 $$ \|z(n \tau) - \mbox{Re }u^n \|_{H^1} +  \|\partial_{t}z(n \tau) - \mbox{Im }u^n \|_{L^2}  \leq C_{T} \tau, \quad 0 \leq n \tau \leq T.$$
          \end{cor} 
 Such an error estimate for a classical first-order scheme for the equation \eqref{kgrO} would require
  (at least) $H^1$ solutions. 
  
   In higher dimensions, our second-order scheme will give more significant improvements.
  \subsubsection{Second-order scheme}
  We shall now study \eqref{kgrO} formulated as \eqref{kgroo} for $x \in \mathbb{T}^3$, 
  and $g(z)= z^2$. We shall again measure the error in $H^1$ so that $\alpha_{0}= 1$.
    In this case, we could take  $a_{0}=1/2$, nevertheless, since we will be forced to take $\alpha_{2}= 7/4$, 
    which is now bigger than $\alpha_{0}$, it is convenient to choose
    $a_{0}= \alpha_{0}= 1$, $a_{1}= \alpha_{2}= 7/4.$
     To get  \eqref{picard}, we observe that with our choice,  by using the embedding 
     $H^{3 \over 4} \subset L^4$  in dimension $3$, we can write   
  \begin{equation}
  \label{fin1} \|  \cnab^{-1} \left( v  w \right) \|_{H^{1}} \lesssim  \|v w \|_{L^2} 
   \lesssim \|v\|_{L^4} \|w \|_{L^4} \lesssim \|v\|_{H^{3 \over 4}} \|w\|_{H^{3 \over 4}}.
   \end{equation}
    This yields \eqref{picard} for $\alpha = a_{0}$.
    For $\alpha \in (a_{0}, a_{1}]$,  $a_{1}=\alpha_{1}= 7/4$, we use that
     $$ \|  \cnab^{-1} \left( v  w \right) \|_{H^{a_{1}}} \lesssim  \|v w \|_{H^{ {1 \over 4} } } $$
  and hence
    by using  the generalized Leibnitz rule, and the embeddings $H^1 \subset L^6, $ $ H^{1 \over 2} \subset L^3$
   we get
 $$ \| v w \|_{H^{  {1 \over 4}}} \lesssim  \| \langle \nabla \rangle^{1 \over 4}v \|_{L^6} \|w\|_{L^3} +  \| \langle \nabla \rangle^{1 \over 4}w \|_{L^6} \|v\|_{L^3} \lesssim \|v \|_{H^{5\over 4}} \|w \|_{H^{1 \over 2}} + \|w \|_{H^{5\over 4}} \|v \|_{H^{1 \over 2}}.$$  
    We have thus obtained that 
  \begin{equation}
  \label{fin0}   \|  \cnab^{-1} \left( v  w \right) \|_{H^{a_{1}}} \lesssim \|v\|_{H^{a_{0}}} \|w\|_{H^{a_{1}}} +  \|w\|_{H^{a_{0}}} \|v\|_{H^{a_{1}}}.
  \end{equation}
       This yields   \eqref{picard}.       
  For \eqref{tech02}, we use that 
  \begin{equation}
  \label{fin2}  \|  \cnab^{-1} \left( v  w \right) \|_{H^{1}} \lesssim  \|v w \|_{L^2} \lesssim \| v \|_{L^3} \| w \|_{L^6} \lesssim 
    \|v \|_{H^{1 \over 2}} \|w \|_{H^1},
    \end{equation}
    since $H^1 \subset L^6$, $H^{1 \over 2} \subset L^3$.
    For \eqref{tech0bisbis2}, we have  that for $k=0, \, 1$,  
    $$ \|\cnab^k( \cnab^{-1} u^2 ) \|_{H^1} \lesssim \| u \nabla u \|_{L^2} \lesssim \|u\|_{L^4} \| \nabla u \|_{L^4}
     \lesssim \|u \|_{H^{7 \over 4}}^2,$$
    $$ \| \cnab^{-1}( u \cnab u )\|_{H^1} \lesssim  \| u \nabla u \|_{L^2}    \lesssim \|u \|_{H^{7 \over 4}}^2, $$
    and that
  \begin{multline*}
  \| \cnab^{-1}( u   \cnab^{-1}( v^2) )\|_{H^1} \| \lesssim \|u \cnab^{-1}v^2 \|_{L^2} 
  \lesssim  \|u\|_{L^3} \| \cnab^{-1} v^2 \|_{L^6} \\ \lesssim \|u \|_{H^{1 \over 2}} \|v^2 \|_{L^2} \lesssim \|u \|_{H^{1 \over 2}}
    \|v\|_{L^4}^2 \lesssim \|u \|_{H^{1 \over 2}} \|v \|_{H^{3 \over 4}}^2.
    \end{multline*}
    In a similar way, we get \eqref{tech0bisbis3} from 
  $$ \| \cnab( \cnab^{-1} ( u \cnab^{-1}(v^2) ) \|_{H^1} \lesssim \|u\|_{H^1} \| \cnab^{-1}(v^2) \|_{L^\infty}
   +  \|u\|_{L^6} \|v \|_{L^6}^2 \lesssim \|u\|_{H^1} \|v\|_{H^1}^2,$$
   since $\alpha_{2}= 7/4$.
   
  To get  \eqref{tech12} \eqref{tech22} (since we have chosen $a_{0}= \alpha_{0}=1$, there is only one set of assumptions
   to check) we can use again \eqref{fin1} and \eqref{fin2}, 
   so that it only remains to estimate the terms of the type $D_{i}f \cdot f$, we use that
   $$ \|\cnab^{-1} (u \cnab^{-1}(vw) \|_{H^1}\lesssim \|u\|_{L^3} \| \cnab^{-1} (vw) \|_{L^6}
    \lesssim \|u\|_{H^{1 \over 2}} \|vw\|_{L^2} \lesssim \|u\|_{H^1} \|v\|_{H^1} \|w\|_{H^1}.$$
    To get  \eqref{tech12}, we can use again \eqref{fin0} since $a_{1}= \alpha_{2}$ with
     the following estimate to handle the terms involving  $D_{i}f \cdot f$
    \begin{multline*}
    \|\cnab^{-1} (u \cnab^{-1}(vw) \|_{H^{7 \over 4}} \lesssim \|u\cnab^{-1}(vw) \|_{H^{3 \over 4}}  \\
     \lesssim \|\cnab^{3 \over 4} u\|_{L^3} \|\cnab^{-1}(vw)\|_{L^6} + \|u\|_{L^6} \|\|\cnab^{-{1 \over 4}}(vw)\|_{L^3}
      \lesssim \|u\|_{H^{5\over 4}} \|v\|_{H^{5 \over 4}}\|w\|_{H^{5\over 4}}.
      \end{multline*}
       
   It only remains to check Assumption \ref{asscom2}. 
    Let us set $B(v,w)= \cnab^{-1} (v w)$, we can easily deduce Assumption \ref{asscom2}, if we can prove that
    \begin{equation}
    \label{cestfini2} \| \mathcal{C}^2 [B, \cnab ] (v,w) \|_{H^1} \lesssim \|v\|_{H^{7 \over 4}}  \|w \|_{H^{7 \over 4}}.
    \end{equation}
     By explicit computation we have
    \begin{multline}
    \label{cestfini} 
     \mathcal{C}^2 [B, \cnab ] (v,w)
      = \cnab^{-1} \left( \cnab^2(vw) - \cnab^2 v\,  w - v\,  \cnab^2 w\right)  \\
     +   2 \cnab^{-1}(\cnab v \cnab w) 
        - 2 \cnab^{-1}\left( \cnab(v \cnab w) - v \cnab^2 w  +   \cnab(w \cnab v) - w \cnab^2 v \right).
     \end{multline}
    Since $\cnab^2 = m^2 - \Delta$, we get from the standard Leibnitz formula that
    \begin{multline*}
      \left \|\cnab^{-1} \left( \cnab^2(vw) - \cnab^2 v\,  w - v \, \cnab^2 w\right) \right\|_{H^1}
     \lesssim \| \cnab^2(vw) - \cnab^2 v\,  w - v\,  \cnab^2 w\|_{L^2}\\ \lesssim \|v\|_{L^4}\| w\|_{L^4}
      + \| \nabla v \|_{L^4} \|\nabla w \|_{L^4}
      \end{multline*}
       and hence from the embedding $H^{3\over 4} \subset L^4$ in dimension $3$, we get
      $$  \left \|\cnab^{-1} \left( \cnab^2(vw) - \cnab^2 v w - v \cnab^2 w\right) \right\|_{H^1}\lesssim  \|v\|_{H^{7 \over 4}}  \|w \|_{H^{7 \over 4}}.$$
    For the second term in the right hand side of \eqref{cestfini}, we use again
    $$ \| \cnab^{-1}(\cnab v \cnab w)  \|_{H^1} \lesssim  \|v\|_{H^{7 \over 4}}  \|w \|_{H^{7 \over 4}}.$$
     The most difficult term is  $\cnab^{-1}\left( \cnab(v \cnab w) - v \cnab^2 w\right) $, for this one, we use again the Kato-Ponce
     inequality to write
   \begin{multline*}
   \|\cnab^{-1}\left( \cnab(v \cnab w) - v \cnab^2 w \right)\|_{H^1} \lesssim
     \| \cnab(v \cnab w) - v \cnab^2 w\|_{L^2} \\
    \lesssim   \| \nabla v\|_{L^4} \| \cnab w \|_{L^4} \lesssim  \|v\|_{H^{7 \over 4}}  \|w \|_{H^{7 \over 4}}.
     \end{multline*}
     We thus obtain that
   \begin{cor}
      For 
       for every $u_{0}  \in H^{7 \over 4}({\mathbb{T}^3})$, $u_{1} \in H^{ {3 \over 4}}(\mathbb{T}^3)$ 
       Let $T>0$ and $z \in \mathcal{C}([0, T], H^{7 \over 4}(\mathbb{T}^3)) \cap \mathcal{C}^1([0, T], H^{{3 \over 4}}(\mathbb{T}^3)) $
  the unique solution of \eqref{kgrO}, with $g(z)= z^2$. Let  $u^n$ denote the numerical solution given by \eqref{num1Intro}.
   Then there exists $C_{T}>0$ such that
 $$ \|z(n \tau) - \mbox{Re }u^n \|_{H^1} +  \|\partial_{t}z(n \tau) - \mbox{Im }u^n \|_{L^2}  \leq C_{T} \tau^2, \quad 0 \leq n \tau \leq T.$$
          \end{cor} 

                               Note that   second-order convergence of  classical  schemes for the Klein--Gordon equation~\eqref{kgrO}   requires (at least) 
  $H^2$ solutions. \red  In a similar way we can treat the wave equation 
$$
\partial_{tt} z - \Delta z  =  h(z) 
$$
by setting $g(z)= h(z) + m^2 z$ for some $m \neq 0$ such that
$
\partial_{tt} z - \Delta z +m^2 z  =  g(z)
$.
 \black 
   
\section{Technical Lemma}
\label{sectiontech}
 In this section, we shall use the Littlewood-Paley decomposition on $\mathbb{T}^d$.
Let us recall some basic facts, we refer to the book \cite{Bahouri-Chemin-Danchin} for example  for the proofs. We take a partition of unity of the form
$$
1={\red  \phi}_{-1}(\xi)+ \sum_{k \geq 0}{\red  \phi}_{k}(\xi)
$$
where ${\red  \phi}_{-1}$ is supported in the ball $\overline{B}(0,1)$ and each ${\red  \phi}_{k}(\xi)= {\red  \phi}(\xi/2^k)$, $k\geq 0$ is supported in the annulus $2^{k-1} \lesssim |\xi | \lesssim 2^{k+1}$. We can then decompose any tempered distribution as
$$
u = \sum_{k \geq -1} \Delta_{k}u , \qquad  \mathcal{F}(\Delta_{k}u) (\xi)= {\red  \phi}_{k}(\xi) \hat u(\xi),
$$
note that we can use the Fourier transform on $\mathbb{R}^d$ or on $\mathbb{T}^d$, in this case
 $\xi \in \mathbb{Z}^d.$
We shall  use the following facts
  (Bernstein inequality):
For every $\sigma \geq 0$ and every $p \in [1,\infty]$, there exist constants $c>0$ and $C>0$ such that for every $k \geq 0$, we have
\beq
\label{bernstein}
c2^{\sigma k}\| ( {\red  \phi}_{k}(-i\nabla)) u\|_{L^p} \leq \| \left|-i\nabla\right|^\sigma ( {\red  \phi}_{k}(-i\nabla) u )\|_{L^p} \leq C 2^{\sigma k} \| ( {\red  \phi}_{k}(-i\nabla) u)\|_{L^p}.
\eeq
For two functions $u, \,  v$, we have Bony's decomposition
\begin{equation}
\label{Bony} uv = T_{u}v + T_{v} u + R(u,v)
\end{equation}
 where we set
 \begin{align*}
 &  T_{a}b= \sum_{q} S_{q-1}a \Delta_{q} b, \quad   S_{q-1}a= \sum_{k \leq q-2} \Delta_{k}a, \quad R(u,v)= \sum_{|p-q| \leq 2} \Delta_{p} a \Delta_{q}b.
 \end{align*}
 We shall first prove
 \begin{lemma}
 \label{lemtech}
  For $u, \, v \in H^{{1 \over 2 } + {d \over 4}}$, we have the estimate
  \begin{equation}
  \label{estlemtech}
   \|| \nabla |(u v) - |\nabla |u v - u \red |\nabla | v\black  \|_{L^2} \lesssim \| u \|_{H^{{1 \over 2} + {d \over 4} }} \| v \|_{H^{{1 \over 2} + {d \over 4} }}.
   \end{equation}
   This estimate also holds with $|\nabla|$ replaced by $\cnab$. 
 \end{lemma}
 \begin{proof}
  The fact that the same estimate holds for   $|\nabla|$ and  $\cnab$ just follows from the observation that
   $ \cnab - |\nabla|$ is a bounded operator. We shall thus prove \eqref{estlemtech}
 We can use the decomposition \eqref{Bony}. The result then follows if we can estimate the following terms
 $$ I= |\nabla| (T_{u}v) - T_{ |\nabla| u} v - T_{u} |\nabla| v, \quad II= R(u, |\nabla |v), \, III= |\nabla| R(u,v)$$
  and symmetric terms.
 
 By using Theorem 2.85 of \cite{Bahouri-Chemin-Danchin}, we get that 
 $$ \| II \|_{L^2} + \|III\|_{L^2} \lesssim \| u \|_{H^{{1 \over 2} + {d \over 4} }} \| v \|_{H^{{1 \over 2} + {d \over 4} }}$$
  so that we only need to estimate $I$.
   By using the frequency localization, we get that
   $$ \|I\|_{L^2}^2 \lesssim \sum_{q} \left\|  |\nabla| (\red S_{q-1}\black  u \Delta_{q} v) - |\nabla| \red S_{q-1}\black u \Delta_{q} v -  \red S_{q-1}\black u |\nabla| \Delta_{q} v\right\|^2_{L^2}.$$
   By using the fourier Transform we need to estimate
 $$ I_{q}= \sum_{k} \left| \sum_{l} m(k-l, l) \widehat{\Delta_{q}v}( k-l) \widehat{\red S_{q-1} u}(l) \right|^2$$
 where $|m(\xi, \eta)|=|  |\xi + \eta| - |\xi| - |\eta | | \lesssim |\xi| |\eta|/ ( |\xi + \eta| + |\xi|+ |\eta|).$
  We remind that due to the frequency localisation in the Littlewood-Paley decomposition, we have
  $  |l| \lesssim 2^{q-1} $ and $ 2^{q-1} \leq |k-l| \leq 2^{q+1}.$
   By using Cauchy-Schwarz and Fubini, we then get that
   $$ I_{q} \lesssim  \sup_{ 2^{q-1} \leq |\xi| \leq 2^{q+1} }  \sum_{|l| \leq 2^{q-2}} {|m(\xi, l)|^2 \over 
    1 + |l|^{ 1 + {d \over 2}}}  \|u\|_{H^{{1 \over 2}+{d \over 4}}}^2 \|\Delta_{q}v \|_{L^2}^2.$$
     Since $|m(\xi,l)| \leq |l|,$ we have that
     $$   \sum_{|l| \leq 2^{q-2}} {|m(\xi, l)|^2 \over 
    1 + |l|^{ 1 + {d \over 2}}} \lesssim    2^{q( 1 + {d \over 2})}.$$
   Therefore, we get that
   $$   \|I\|_{L^2}^2 \lesssim \sum_{q} I_{q} \lesssim 2^{q( 1 + {d \over 2})}  \|u\|_{H^{{1 \over 2}+{d \over 4}}}^2 \|\Delta_{q}v \|_{L^2}^2
    \lesssim  \|u\|_{H^{{1 \over 2}+{d \over 4}}}^2 \|v\|_{H^{{1 \over 2}+{d \over 4}}}^2$$
    where the last inequality comes from the Bernstein inequality and the almost orthogonality of the terms.
    This concludes the proof.
 
 \end{proof} 
 
 \begin{lemma}
 \label{lemtech2}
 For $d=1$, 
  $u \in H^{ {3 \over 4}}(\mathbb{T})$, we have
  $$ \|\cnab \sin u  - \cos u \cnab u \|_{L^2} \lesssim C( \|u\|_{H^{{3 \over 4}}}).$$
 \end{lemma}
 \begin{proof}  
 By using for \red example \black \cite{Metivier} Theorem 5.2.5 (still valid on the torus), we have that for $u \in H^s$, 
   $s>d/2$ (we take $s=3/4$), $ \sin u - T_{\cos u } u \in  H^{2s - {d \over 2}}$.
    We thus get that
    $$  \cnab \sin u = \cnab ( T_{\cos u } u )+ \mathcal{R}_{1}$$
     with $\mathcal{R}_{1} \in H^{2s- { 3 \over 2}}.$
  Next, we can write that
  $$ \cos u \cnab u= T_{\cos u} \cnab u  + T_{\cnab u} \cos u + R(\cos u, \cnab u).$$
   By using  again Theorem 2.85  of \cite{Bahouri-Chemin-Danchin}, we have
    that
    $$ \|R(\cos u, \cnab u ) \|_{L^2} \lesssim \|\cos u \|_{H^{3 \over 4}} \| \cnab u \|_{H^{-{1 \over 4}}} \lesssim  C( \|u\|_{H^{{3 \over 4}}}).$$
     By using also   Theorem 2.82 of \cite{Bahouri-Chemin-Danchin}, we also have 
     $$ \|T_{\cnab u} \cos u\|_{L^2} \lesssim  \| \cnab^{1\over 4} u \|_{L^\infty} \| \cos u \|_{H^{3 \over 4}}
 \lesssim C( \|u\|_{H^{{3 \over 4}}}).$$
 Therefore, it only remains to estimate
 $$ \cnab ( T_{\cos u } u ) - T_{\cos u} \cnab u.$$
 By frequency localization, we have that
 $$  \|\cnab ( T_{\cos u } u ) - T_{\cos u} \cnab u\|_{L^2}^2
   \lesssim \sum_{k}  \| \cnab (S_{k-1} \cos u \, \Delta_{k} u)- S_{k-1} \cos u \cnab \Delta_{k} u \|_{L^2}^2, $$
   therefore, by using Lemma \ref{lemtech}, we get that
 \begin{multline*}
  \|\cnab ( T_{\cos u } u ) - T_{\cos u} \cnab u\|_{L^2}^2
    \lesssim \sum_{k} \|S_{k-1} \cos u \|_{H^{3\over 4}}^2 \| \Delta_{k} u \|_{H^{3\over 4}}^2\\
     \lesssim  \|\cos u \|_{H^{3 \over 4}}^2 \sum_{k} \| \Delta_{k} u \|_{H^{3\over 4}}^2
      \lesssim    \|\cos u \|_{H^{3 \over 4}}^2 \|u\|_{H^{3\over 4}}^2.
    \end{multline*}
    We conclude by using again that $$ \|\cos u \|_{H^{3 \over 4}} \lesssim 1 +  \|u\|_{H^{3\over 4}}.$$ 
    This concludes the proof.  
\end{proof}

\section*{Acknowledgement}
This project has received funding from the European Research Council (ERC) under the European Unions Horizon 2020 research and innovation programme (grant agreement No. 850941).

\end{document}